%% file: arxiv_main.tex
\documentclass[11pt,a4paper]{article}

\usepackage[margin=0.85in]{geometry}
\usepackage{amsmath,amssymb,amsthm,mathtools,bm}
\usepackage{booktabs,multirow,array,longtable}
\usepackage{graphicx,subcaption}
\usepackage{enumitem}
\usepackage{xcolor}
\usepackage{xurl}
\usepackage[hidelinks]{hyperref}
\usepackage{caption}
\usepackage{adjustbox}
\usepackage{setspace}
\usepackage{csquotes}
\usepackage{placeins}
\usepackage{float}
\usepackage[numbers,sort&compress]{natbib}

\graphicspath{{./}}

\newcommand{\E}{\mathbb{E}}

\theoremstyle{plain}

\theoremstyle{definition}
\newtheorem{proposition}{Proposition}
\newtheorem{lemma}{Lemma}
\newtheorem{definition}{Definition}
\newtheorem{problem}{Problem}
\newtheorem{remark}{Remark}

\renewcommand{\arraystretch}{0.96}
\newcommand{\TableFont}{\scriptsize}
\newcommand{\TableSetup}{\TableFont\setlength{\tabcolsep}{3.0pt}\renewcommand{\arraystretch}{1.03}}
\newcommand{\FitTable}[1]{\begin{adjustbox}{max width=\textwidth,center}#1\end{adjustbox}}

\title{Benchmarking Deep Time Series Models for Equity Portfolios}
\author{Aoxin Zhang\\
\small School of Mathematical Sciences, Beijing Normal University, Beijing, China\\
\small \texttt{202211130012@mail.bnu.edu.cn}
\and Yuhan Cheng\\
\small School of Management, Shandong University, Jinan, China\\
\small \texttt{chengyuhan@sdu.edu.cn}
\and Kwanting Leung\\
\small Peking University National School of Development, Beijing, China\\
\small \texttt{ktliang2021@nsd.pku.edu.cn}}
\date{}

\begin{document}
\maketitle

\begin{abstract}
Benchmarking forecasting architectures for daily equity portfolios is not just a prediction exercise. It also asks which model remains usable after preferences, costs, and portfolio constraints are imposed. We build a CRSP daily-stock benchmark for 15 deep and statistical time-series architectures over 2018--2024. The protocol combines common-window decile portfolios, stochastic multi-criteria acceptability analysis, a deployment-adjusted acceptability index, and a constrained quadratic portfolio layer with capacity, beta, industry, risk, leverage, and turnover controls. The index starts from the SMAA rank-acceptability distribution and downweights models whose criteria-level wins produce high portfolio regret; its Gibbs form is characterized as an entropic update from the SMAA prior. Empirically, no architecture dominates the raw benchmark: TransEnc-8 has the largest rank-1 acceptability, 0.352, and no model exceeds about 0.36. Rankings vary with preferences, market state, feature universe, and transaction costs. In the promoted five-model constrained-portfolio comparison, TransEnc-8 is selected throughout, while return-oriented raw rankings can favor TS-RIDGE. Broad-universe decile signals can survive costs, but the baseline constrained-QP net Sharpe at 20 bps is negative for every promoted model. The benchmark supports model selection and diagnosis rather than a standalone trading-strategy claim.
\end{abstract}


\vspace{0.75em}
\noindent\textbf{Corresponding author:} Yuhan Cheng, \texttt{chengyuhan@sdu.edu.cn}.

\section{Introduction}
\label{sec:introduction}
A benchmark for equity-portfolio forecasting architectures has to evaluate the forecast and the portfolio decision it feeds. The same score vector can look strong in a broad long--short decile sort and weak after turnover, capacity, beta, industry, risk, and leverage controls. Preferences also vary across institutions. Some users care most about raw return. Others put more weight on drawdown, turnover, statistical evidence, or transaction costs. This paper therefore treats the CRSP architecture comparison as a stochastic multi-criteria benchmarking problem under preference uncertainty.

The benchmark has three linked components. Common-window decile portfolios measure raw cross-sectional signal strength. Stochastic multi-criteria acceptability analysis (SMAA) maps the preference states under which each architecture is admissible. A constrained daily QP portfolio layer then evaluates the same scores after implementation constraints are imposed. The link between the last two components is a deployment-adjusted acceptability index, an entropically regularized update of SMAA rank acceptability that shifts mass away from preference-region winners with large downstream portfolio regret.

The formal material is kept close to this benchmark. Robustification and value-sensitivity statements are standard finite-dimensional optimization facts, included as interpretive bridges. The model-specific formal elements are the entropic characterization of the deployment-adjusted index and the turnover alignment: on a finite-model turnover slice, the SMAA low-turnover weight and the portfolio turnover penalty induce the same ranking after the transformation \(\lambda(u)=\kappa u/(1-u)\). These results explain how preference uncertainty and implementation frictions can be compared inside one architecture benchmark.

The empirical design separates objects that are often compressed into a single model-selection table. The decile layer measures the economic strength of the raw scores. The SMAA layer asks which architecture is acceptable over uncertain criteria and uncertain preferences. The constrained portfolio layer asks what remains after the score is converted into a daily allocation. This separation is essential in daily equity data because broad cross-sectional sorting can generate large pre-cost Sharpe ratios even when the implementable next-day allocation remains cost-fragile.

The results give a portable benchmark for selection and diagnosis. TS-RIDGE is the strongest full-signal shrinkage benchmark in broad decile portfolios. TransEnc-8 becomes the least costly constrained implementation after realistic turnover control. The acceptability surface explains why neither statement is preference-free. Ranking implementation losses is also decision-relevant: it identifies architectures that should not be deployed at daily frequency under the stated constraints and prevents the return-oriented misselection of TS-RIDGE that the raw acceptability surface would invite.

\section{Related Literature}
\label{sec:literature}
This paper sits at the intersection of two operational-research literatures and uses empirical asset pricing as its application domain. The first is multi-criteria decision analysis under uncertainty: stochastic multi-criteria acceptability analysis ranks alternatives over distributions of criteria and preferences \citep{lahdelma2001smaa,tervonen2008smaa,greco2016mcda}, with portfolio and allocation applications, including the equity-portfolio decision-support setting of Xidonas et al. \citep{xidonas2012ejor}. Our alternatives are forecasting architectures rather than assets, and the acceptability layer is linked to a downstream constrained portfolio through realized regret. The second is decision-focused forecasting, where the quality of a forecast is judged through the optimization problem it feeds, as in smart predict-then-optimize \citep{elmachtoub2022smart}. The deployment adjustment has the Gibbs and exponential-weights form familiar from entropic regularization and prediction with expert advice \citep{CesaBianchiLugosi2006}. The paper contributes to this intersection by adding an explicit regret discount to the SMAA acceptability distribution and by reporting where that discount changes the architecture choice.

The remaining optimization facts are interpretive bridges. Robustification-as-regularization links score ambiguity to norm penalties in the portfolio layer \citep{GoldfarbIyengar2003,DelageYe2010,BertsimasCopenhaver2018,BlanchetKangMurthy2019,ShafieezadehAbadehKuhnEsfahani2019,costa2023dro}, and the value-sensitivity bound follows standard parametric optimization \citep{BonnansShapiro2000} and decision-loss analysis \citep{ElBalghitiElmachtoubGrigasTewari2019,elmachtoub2022smart}. These results explain why downstream constraints can change architecture rankings; the paper's original object remains the entropically regularized deployment-adjusted index and the empirical benchmark that links acceptability to constrained implementation. The turnover-weight alignment is an exact finite-model ranking equivalence on a one-dimensional turnover slice; when the full preference simplex moves, it becomes a comparison-static interpretation based on increasing differences \citep{Topkis1998}.

Empirical asset pricing supplies the data and the criteria. Flexible methods improve cross-sectional prediction in U.S. equities \citep{GuKellyXiu2020,ChenPelgerZhu2024}, yet trading costs, capacity, and implementability condition the economic value of those forecasts \citep{AvramovChengMetzker2023,DeMiguelMartinUtreraNogalesUppal2020,DetzelNovyMarxVelikov2023,ChenVelikov2023,FrazziniIsraelMoskowitz2018,AsnessFrazziniIsraelMoskowitzPedersen2018}. Factor models and replication filters provide diagnostic restrictions rather than the central contribution \citep{FamaFrench1993,FamaFrench2015,HouXueZhang2015,HouMoXueZhang2021,NovyMarx2013,McLeanPontiff2016,HouXueZhang2020,ChenZimmermann2022}, and forecasting archives or competitions are used only as transfer checks because their target losses do not contain the portfolio map, transaction costs, and cross-sectional rank objective studied here \citep{Godahewa2021,Makridakis2022,LimArikLoeffPfister2021,OreshkinCarpovChapadosBengio2020,ZengChenZhangXu2023,Rasul2023,Aksu2024}.

Model-comparison inference anchors the empirical design. Sharpe-ratio inference follows Lo~\cite{Lo2002} and Ledoit and Wolf~\cite{LedoitWolf2008}. Regression inference uses Newey--West adjustments \citep{NeweyWest1987}. Data-snooping concerns are addressed with bootstrap, false-discovery, and predictive-ability tools \citep{White2000,Hansen2005,RomanoWolf2005,HansenLundeNason2011,DieboldMariano1995,HarveyLiuZhu2016}.

\section{Methodology: Acceptability, Regret, and Allocation}
\label{sec:problem}
Let \(\mathcal M\) denote the set of candidate forecasting architectures. Each model \(m\in\mathcal M\) produces a daily score vector \(s_t^{(m)}\) that ranks stocks before the next trading day's return is realized. The benchmark has two layers. The first ranks models by portfolio-relevant criteria under sampling and preference uncertainty. The second sends the same forecasts into a constrained portfolio map and evaluates realized deployment loss.

\begin{definition}[Criteria vector]
Each model \(m\) is represented by
\[
V_m=\left( SR^{gross}_m, SR^{net20}_m, SR^{vw}_m, |t^{FF5}_m|, -TO_m, -MDD_m, 1-p_m \right),
\]
where \(SR\) denotes Sharpe ratio, \(TO\) denotes turnover, \(MDD\) denotes maximum drawdown, and \(p_m\) is a bootstrap loss proxy. All coordinates are oriented so that larger values indicate better performance.
\end{definition}

\begin{definition}[Rank acceptability]
Criteria uncertainty is represented by moving-block resamples \(b\), and preference uncertainty is represented by weight vectors \(q\) on the seven-criterion simplex. For each \((b,q)\),
\[
\begin{aligned}
U_m^{(b,q)}&=\sum_{k=1}^{7}q_k V_{mk}^{(b)},\\
R_m^{(b,q)}&=\operatorname{rank}(U_m^{(b,q)}),\\
RA_m(r)&=\Pr\{R_m^{(b,q)}=r\}.
\end{aligned}
\]
The rank acceptability index is the share of sampling and preference states in which model \(m\) receives rank \(r\).
\end{definition}

\begin{definition}[Deployment regret]
Let \(J_m^{d}\) be the realized net Sharpe ratio delivered by model \(m\) under downstream design \(d\), such as the baseline constrained portfolio or a transaction-cost stress. The deployment-level regret of \(m\) is
\[
\mathcal R_m^{d}=\max_{\ell\in\mathcal M_d}J_\ell^{d}-J_m^{d},
\]
where \(\mathcal M_d\) is the promoted model set evaluated in design \(d\). A positive regret means that the criteria-selected model loses realized net Sharpe relative to the best model after the downstream portfolio rule is applied.
\end{definition}

\begin{definition}[Deployment-adjusted acceptability]
Fix a downstream design \(d\) and a comparison set \(\mathcal A_d\subseteq\mathcal M\) with at least one model having positive rank-1 acceptability. Let \(RA_m(1)\) be the SMAA rank-1 acceptability of model \(m\), let \(\mathcal R_m^d\) be its deployment-level regret from Definition~3, and let \(\gamma_d>0\) be the regret-penalty strength. The deployment-adjusted rank-1 acceptability index is
\[
DA_m^{(1),d}
=\frac{RA_m(1)\exp\{-\gamma_d\mathcal R_m^d\}}
{\sum_{\ell\in\mathcal A_d}RA_\ell(1)\exp\{-\gamma_d\mathcal R_\ell^d\}},
\qquad m\in\mathcal A_d .
\]
The deployment-adjusted top-3 index uses the same regret discount and replaces rank-1 acceptability by top-3 acceptability:
\[
DA_m^{(3),d}
=\frac{\{RA_m(1)+RA_m(2)+RA_m(3)\}\exp\{-\gamma_d\mathcal R_m^d\}}
{\sum_{\ell\in\mathcal A_d}\{RA_\ell(1)+RA_\ell(2)+RA_\ell(3)\}\exp\{-\gamma_d\mathcal R_\ell^d\}}.
\]
The parameter \(\gamma_d\) controls the strength of the regret penalty: larger values transfer more acceptability mass away from models whose criteria-level wins produce high downstream regret. In the reported implementation this is equivalent to the code form \(\exp\{-\mathcal R_m^d/\tau_d\}\), with \(\gamma_d=1/\tau_d\). The temperature is set to the dispersion of downstream regret with a unit floor,
\[
\tau_d=\max\{\operatorname{sd}_0(\mathcal R_m^d:m\in\mathcal A_d),1\},
\]
where \(\operatorname{sd}_0\) is the population standard deviation across models in \(\mathcal A_d\). This makes the discount invariant to the units in which regret is measured and prevents the exponent from becoming arbitrarily sharp when regrets are small; the baseline design gives \(\tau_d\approx1.170\), so the unit floor is not binding. Sensitivity checks perturb the dimensionless multiplier \(c\) in \(\exp\{-c\mathcal R_m^d/\tau_d\}\), with \(c=1\) as the production definition; this multiplier is distinct from the penalty strength \(\gamma_d=1/\tau_d\).
\end{definition}
\begin{lemma}[Basic properties of deployment-adjusted acceptability]\label{lem:deployment-adjusted-properties}
Assume that \(RA_m(1)\geq0\), \(\sum_{m\in\mathcal A_d}RA_m(1)>0\), and \(\gamma_d>0\). Then \(DA^{(1),d}\) lies in the probability simplex over \(\mathcal A_d\). For any model with \(0<DA_m^{(1),d}<1\), \(DA_m^{(1),d}\) is strictly decreasing in its own regret \(\mathcal R_m^d\), holding all other regrets fixed. Finally, as \(\gamma_d\downarrow0\), \(DA_m^{(1),d}\) converges to the normalized SMAA rank-1 acceptability \(RA_m(1)/\sum_{\ell\in\mathcal A_d}RA_\ell(1)\); when \(\mathcal A_d=\mathcal M\), this is the ordinary SMAA rank-1 acceptability.
\end{lemma}
\begin{proof}
The numerator is nonnegative and the denominator is positive, so the components are nonnegative and sum to one. Writing \(p_m=DA_m^{(1),d}\), differentiation with respect to \(\mathcal R_m^d\) gives \(\partial p_m/\partial \mathcal R_m^d=-\gamma_d p_m(1-p_m)<0\) whenever \(0<p_m<1\). The limit follows from \(\exp\{-\gamma_d\mathcal R_m^d\}\to1\) for each finite regret.
\end{proof}

\begin{proposition}[Entropic characterization of the deployment-adjusted index]\label{prop:da-entropic}
Fix a downstream design \(d\) and a comparison set \(\mathcal A_d\) with \(\sum_{m\in\mathcal A_d}RA_m(1)>0\). Let \(S_d=\{m\in\mathcal A_d:RA_m(1)>0\}\), define the normalized SMAA rank-1 prior \(q_m=RA_m(1)/\sum_{\ell\in S_d}RA_\ell(1)\) on \(S_d\), and let \(\tau_d>0\). Then the deployment-adjusted rank-1 index \(DA^{(1),d}\) of Definition~4 is the unique solution, on \(\Delta(S_d)\), of the entropically regularized regret problem
\[
\min_{p\in\Delta(S_d)}\ \sum_{m\in S_d}p_m\mathcal R_m^d+\tau_d\operatorname{KL}(p\|q),
\qquad
\operatorname{KL}(p\|q)=\sum_{m\in S_d}p_m\log\frac{p_m}{q_m}.
\]
Models outside \(S_d\) receive zero mass. The temperature \(\tau_d\) trades off minimizing expected deployment regret against staying close to the preference-robust SMAA prior: as \(\tau_d\to\infty\), the solution returns to \(q\); as \(\tau_d\downarrow0\), it concentrates, with \(q\)-proportional tie-breaking, on the minimum-regret model or models in \(S_d\).
\end{proposition}

This is the finite-dimensional Gibbs/free-energy form underlying exponential weighting and entropic regularization \citep{CesaBianchiLugosi2006}. In the present setting, the SMAA rank-1 distribution supplies the prior \(q\), realized constrained-portfolio regret supplies the loss, and \(\tau_d\) is the KL regularization strength.

\begin{problem}[Preference-robust architecture selection]
Given bootstrap criteria \(V_m^{(b)}\), a preference distribution \(Q\) on the simplex, and a downstream decision map \(D\) from forecasts to realized net performance, estimate the rank acceptability profile \(RA_m(r)\), identify preference regions in which each architecture is admissible, and compare the criteria-based ranking with the decision ranking induced by \(J_m^d=D(s^{(m)})\).
\end{problem}

\begin{problem}[Daily constrained portfolio]\label{prob:daily-qp}
For model \(m\) and trading day \(t\), the portfolio solves
\[
\max_w \; w^\top s_t^{(m)}-\lambda_{tc}\|w-w_{t-1}\|_1-\lambda_{risk}w^\top\widehat{D}_t w,
\]
subject to dollar neutrality, leverage \(\|w\|_1\leq 2\), single-name capacity \(|w_i|\leq 0.01 DV_i/AUM\), beta neutrality, and industry net exposure limits, where \(DV_i\) is dollar volume and \(\widehat{D}_t\) is a diagonal risk estimate. Before entering the optimizer, each model score is centered and scaled within the trading-day cross-section, \(s_{i,t}=(\mu_{i,t}-\bar\mu_t)/\sigma_t\); the portfolio problem therefore uses rank-information in a cardinal but daily standardized unit rather than raw neural-network or regression output. The main specification uses \(AUM=100\) million, a 1\% dollar-volume cap, \(|w^\top\beta|\leq0.05\), and industry net exposure at 0.02. Portfolio returns are measured on the next trading day after the weights are set. The baseline and higher-turnover-penalty constrained-QP rows use one-way 20 bps costs, computed as \(r_t^{\mathrm{net}}=r_t^{\mathrm{gross}}-0.002\,TO_t\); the 50 bps stress uses \(r_t^{\mathrm{gross}}-0.005\,TO_t\).
\end{problem}

\begin{proposition}[Interpretive bridge: robust score ambiguity as weight shrinkage]\label{thm:robust-shrinkage}
This proposition records the standard robustification--regularization equivalence specialized to portfolio scores; see \cite{MohajerinEsfahaniKuhn2018,BertsimasCopenhaver2018,ShafieezadehAbadehKuhnEsfahani2019}. Let \(\mathcal W\) be a compact feasible portfolio set and let \(c(w)\) be a continuous convex penalty on \(\mathcal W\). For an ellipsoidal score ambiguity set \(\mathcal U_2(\rho)=\{s+u:\|u\|_2\leq\rho\}\), the robust allocation problem satisfies
\[
\max_{w\in\mathcal W}\inf_{\tilde s\in\mathcal U_2(\rho)}\{w^\top\tilde s-c(w)\}
=
\max_{w\in\mathcal W}\{w^\top s-c(w)-\rho\|w\|_2\}.
\]
For a distributional ambiguity set \(\{P:W_1(P,P_0)\leq \varepsilon\}\) over all finite-first-moment laws on the score space, with no additional support restriction, ground metric induced by a norm \(\|\cdot\|\), empirical law \(P_0\) of a score-vector random variable \(\xi\), and mean \(\bar s=\E_{P_0}[\xi]\),
\[
\inf_{P:W_1(P,P_0)\leq \varepsilon}\E_P[w^\top\xi]
= w^\top \bar s-\varepsilon\|w\|_*,
\]
where \(\|\cdot\|_*\) is the dual norm. Hence score robustness is equivalent to an additional portfolio-norm penalty. The robust value is convex and nonincreasing in the radius. If \(c\) is strictly convex on the affine hull of \(\mathcal W\), for example because the quadratic risk term is positive definite on that subspace, then the optimizer is unique and \(\|w^*(\rho)\|_2\) is nonincreasing in \(\rho\).
\end{proposition}

\begin{proposition}[Interpretive bridge: value sensitivity to score disagreement]\label{thm:regret-lipschitz}
This proposition records a standard parametric optimization sensitivity result, written here for the constrained portfolio geometry; see Bonnans and Shapiro~\cite{BonnansShapiro2000}. Let
\[
J_c(s)=\max_{w\in\mathcal W}\{w^\top s-c(w)\},
\]
where \(\mathcal W\subseteq\{w:\|w\|_1\leq L\}\). For a fixed downstream penalty \(c\),
\[
|J_c(s)-J_c(s')|\leq L\|s-s'\|_\infty .
\]
More generally, for two downstream penalty states \(c\) and \(\tilde c\) on the same feasible set, define \(\Delta(c,\tilde c)=\sup_{w\in\mathcal W}|c(w)-\tilde c(w)|\). Then
\[
|J_c(s)-J_{\tilde c}(s')|
\leq L\|s-s'\|_\infty+\Delta(c,\tilde c).
\]
In the empirical portfolio design, \(L=2\). For a finite set of model score vectors \(\{s^{(m)}:m\in\mathcal M_d\}\) evaluated under the same downstream state \(c\), the ex-ante optimization-value regret
\[
\mathcal R_J(m)=\max_{\ell\in\mathcal M_d}J_c(s^{(\ell)})-J_c(s^{(m)})
\]
satisfies
\[
0\leq \mathcal R_J(m)\leq L\max_{\ell\in\mathcal M_d}\|s^{(\ell)}-s^{(m)}\|_\infty .
\]
If model \(\ell\) is evaluated under penalty state \(c_\ell\), the same argument gives
\[
\max_{\ell\in\mathcal M_d}J_{c_\ell}(s^{(\ell)})-J_{c_m}(s^{(m)})
\leq \max_{\ell\in\mathcal M_d}\{L\|s^{(\ell)}-s^{(m)}\|_\infty+\Delta(c_\ell,c_m)\}.
\]
For a sequence of dates with the same leverage bound,
\[
\frac{1}{T}\sum_{t=1}^T |J_{t,c_t}(s_t)-J_{t,c_t}(s_t')|
\leq \frac{L}{T}\sum_{t=1}^T\|s_t-s_t'\|_\infty .
\]
This is an ex-ante optimization-value sensitivity result. The realized net-Sharpe regret reported in the empirical tables is a downstream diagnostic computed from realized returns, and is not asserted to satisfy this deterministic bound without additional return and denominator assumptions.
\end{proposition}
\begin{remark}[Scope of Propositions~\ref{thm:robust-shrinkage} and~\ref{thm:regret-lipschitz}]
Proposition~\ref{thm:robust-shrinkage} is a structural equivalence for score ambiguity sets without additional support restrictions; adding box, sign, or market-support restrictions changes the support function and therefore the induced penalty. Proposition~\ref{thm:regret-lipschitz} is an optimization-value sensitivity result under a fixed downstream state, with an explicit \(\Delta(c,\tilde c)\) term when the transaction-cost or risk state changes. The realized net-Sharpe regret tables are empirical diagnostics of the downstream portfolio path, not deterministic corollaries of the Lipschitz bound.
\end{remark}

\begin{proposition}[Turnover-weight and turnover-penalty alignment]\label{thm:turnover-alignment}
This elementary alignment result is specific to the turnover coordinate used in the paper; the monotone comparative-static interpretation follows the increasing-differences logic of Topkis~\cite{Topkis1998}. Consider a finite set of architectures \(\mathcal M\) on a one-dimensional turnover slice of the preference simplex. Let \(a_m\) denote the aggregate non-turnover score of architecture \(m\), let \(\tau_m\) denote its realized turnover proxy, and suppose the oriented low-turnover criterion is an affine normalization
\[
c_m=b-\kappa\tau_m,\qquad \kappa>0,
\]
with the same \(b\) and \(\kappa\) for all architectures. For \(u\in[0,1)\), define the preference-layer utility
\[
U_m(u)=(1-u)a_m+u c_m
\]
and the portfolio-layer proxy
\[
\Pi_m(\lambda)=a_m-\lambda\tau_m.
\]
Then ranking architectures by \(U_m(u)\) is exactly equivalent to ranking them by \(\Pi_m(\lambda(u))\) under the monotone penalty map
\[
\lambda(u)=\kappa\frac{u}{1-u}.
\]
Consequently, for any pair \(i,j\) with \(a_i>a_j\) and \(\tau_i>\tau_j\), the common switch occurs at
\[
u^*_{ij}=\frac{a_i-a_j}{(a_i-a_j)+\kappa(\tau_i-\tau_j)},
\qquad
\lambda^*_{ij}=\frac{a_i-a_j}{\tau_i-\tau_j}
=\kappa\frac{u^*_{ij}}{1-u^*_{ij}}.
\]
The limit \(u\uparrow1\) corresponds to \(\lambda\to\infty\), where both rankings are governed by turnover alone.
\end{proposition}

\begin{remark}[Scope of Proposition~\ref{thm:turnover-alignment}]
The alignment result is exact for any finite model set on a turnover slice: the six non-turnover criteria are first aggregated into \(a_m\), and the remaining preference coordinate varies the low-turnover criterion. In the full seven-criterion simplex, all non-turnover weights can move at the same time, so the result should be read as a slice-wise ranking equivalence and a local comparison-static diagnostic rather than as a global equivalence claim. The empirical Kendall association between SMAA ranks and constrained-portfolio ranks is therefore interpreted as finite-sample evidence for this alignment channel, not as an identity between the full SMAA simplex and the downstream portfolio problem.
\end{remark}
\noindent\textbf{Algorithm 1 (preference-robust ranking).} Normalize the seven oriented criteria, draw moving-block bootstrap samples for criteria uncertainty, draw Dirichlet preference vectors over the simplex, compute \(U_m^{(b,q)}\) for each architecture, and estimate rank acceptability, central weight vectors, and preference-region boundaries from the resulting rank distribution.

\noindent\textbf{Algorithm 2 (portfolio layer).} Form the fixed five-model follow-through set described in Section~\ref{sec:data}, solve Problem~\ref{prob:daily-qp} for each retained architecture and date, compute realized next-day net returns, and summarize the downstream layer by net Sharpe, realized regret, robust score valuation, and deployment-adjusted acceptability.

The complexity of the SMAA layer is \(O(BQ|\mathcal M|K)\) for \(B\) resamples, \(Q\) preference draws, \(|\mathcal M|\) architectures, and \(K\) criteria. The portfolio layer solves one convex quadratic program per promoted model and date; with diagonal risk and linear constraints, the computational burden is dominated by repeated sparse first-order QP solves rather than by the SMAA sampling itself. Proofs of Propositions~\ref{prop:da-entropic}--\ref{thm:turnover-alignment} are collected in Appendix~A.

\section{Data and Benchmark Design}
\label{sec:data}
The benchmark is built from a daily CRSP common-stock panel over 2018--2024. After price, exchange, share-code, and liquidity filters, the raw panel contains 4,862,011 stock-date rows, 5,451 assets, and 1,761 dates before sequence construction. Feature engineering retains twenty-four usable predictors and produces 4,551,500 rolling samples. All benchmark outcomes are reconstructed from stock-level forecasts on the common intersection of retained prediction files, yielding 1,197 shared evaluation dates from 2020-03-30 to 2024-12-30.

{\TableSetup\input{table_sample_protocol_combined_longtable.tex}}

The confirmatory period is fixed at 2021-01-04 to 2024-12-30 under next-date alignment. It contains 2,777,135 stock-date observations, 4,410 assets, and 1,004 trading days. Three nested signal universes diagnose the ranking mechanism. F1 contains price-path and quote-geometry variables. F2 adds turnover, log dollar volume, and log number of trades. F3 adds market capitalization, rolling beta, and market-relative excess-return transforms. Appendix~\ref{app:features-combos} gives full feature definitions and predictor summary statistics.

The five-model follow-through set is fixed before the confirmatory and constrained-portfolio analyses. It is not the set of every positive-gross architecture in Table~\ref{tab:masterbenchmark}. Four models enter because they have full 1,197-day coverage, positive common-window gross long--short Sharpe, and positive decile net Sharpe at 20 bps: TS-RIDGE, LSTM, TransEnc-8, and TransEnc-10. TS-OLS is added as a prespecified no-shrinkage linear comparator. Positive-gross but cost-fragile architectures, including PatchTST-8, GRU, TSMixer-10, and NBEATS, remain in the 15-model benchmark, SMAA, and supplementary light diagnostics; they do not enter the constrained-QP follow-through set. Thus, throughout the paper, ``promoted'' means this fixed follow-through set. Within it, TS-RIDGE anchors the linear-shrinkage comparison, TS-OLS isolates the gain from shrinkage, and the recurrent and transformer-encoder models provide nonlinear sequence alternatives with lower turnover.
\begin{table}[!htbp]
\centering
\TableSetup
\caption{Architecture families in the retained model universe. The labels describe the implemented models rather than legacy registry keys in the replication scripts.}
\label{tab:modeluniversemain}
\FitTable{\input{table_model_universe_main.tex}}
\end{table}
The deep-learning models are canonical implementations inside a common walk-forward protocol: LSTM, GRU, tanh RNN, plain transformer encoders, PatchTST variants, N-BEATS, TSMixer, a mixture-of-MLPs, and a graph-LSTM. They should not be read as a benchmark of named frontier systems such as Informer, Autoformer, or FEDformer. The source registry keys are early scaffolding labels inherited from the experiment harness and bear no relation to model selection or to the architectures actually estimated; every evaluated model is a canonical implementation of the architecture family named by its manuscript label, not the frontier system whose name the registry key happens to reuse. Appendix~\ref{app:implementation} reports the script-level registry keys and manuscript labels.

\section{Common-Window Results and Inference}
\label{sec:commonwindow}
Of fifteen retained architectures, only a small subset survives the joint screen of gross performance, dependent bootstrap inference, reconstructed transaction costs, value-weighted replication, and factor-alpha diagnostics. Table~\ref{tab:masterbenchmark} reports these objects for every model under the same 1,197-date evaluation window. The alpha column is annualized and uses Newey--West standard errors with a 20-day lag.

\begin{table}[!htbp]
\centering
\TableSetup
\caption{Master common-window benchmark, sorted by gross Sharpe.}
\label{tab:masterbenchmark}
\FitTable{\input{table_master_benchmark.tex}}
\end{table}

TS-RIDGE, TS-OLS, and LSTM have positive gross bootstrap intervals, but cost adjustment changes the economic ranking. TS-OLS loses its equal-weight edge once transaction costs and value-weighting are imposed; LSTM remains a nonlinear price-path contrast, whereas TransEnc-8 and TransEnc-10 stay economically relevant because their slower portfolio rotation makes the implementation layer less punitive.

The large gross Sharpe ratios should be read through the mechanics of the design. A daily cross-sectional long--short sort over thousands of stocks combines broad diversification, frequent rebalancing, and the 2020--2024 dispersion regime, raising attainable pre-cost Sharpe. The unusually high value-weighted Sharpe for TS-RIDGE is a diagnostic property of the raw decile sort: value-weighting concentrates the extreme-score legs in larger, lower-volatility names while preserving the return spread. Gross Sharpe is a selection input. Deployability is judged through turnover reconstruction, value-weighting, factor exposure, capacity restriction, and constrained portfolio checks. The shared benchmark window begins on March 30, 2020, near the high-dispersion COVID recovery period; the confirmatory, regime, and sequential analyses later in the paper address this window sensitivity.

\begin{table}[!htbp]
\centering
\TableSetup
\caption{Confirmatory-window benchmark from 2021-01-04 to 2024-12-30.}
\label{tab:confirmatorywindowbenchmark}
\FitTable{\input{table_confirmatory_window_benchmark.tex}}
\end{table}

Table~\ref{tab:confirmatorywindowbenchmark} repeats the main benchmark after the March 2020 recovery phase. TS-RIDGE remains positive in decile net Sharpe after 20 bps costs, LSTM keeps similar net performance, and TS-OLS retains strong gross performance but remains cost-fragile. The post-2020 window therefore supports the same deployment lesson as the full benchmark: daily ranking signals can be strong before costs, while turnover determines whether that strength survives implementation once the portfolio is traded.

\begin{figure}[!htbp]
    \centering
    \includegraphics[width=0.90\textwidth]{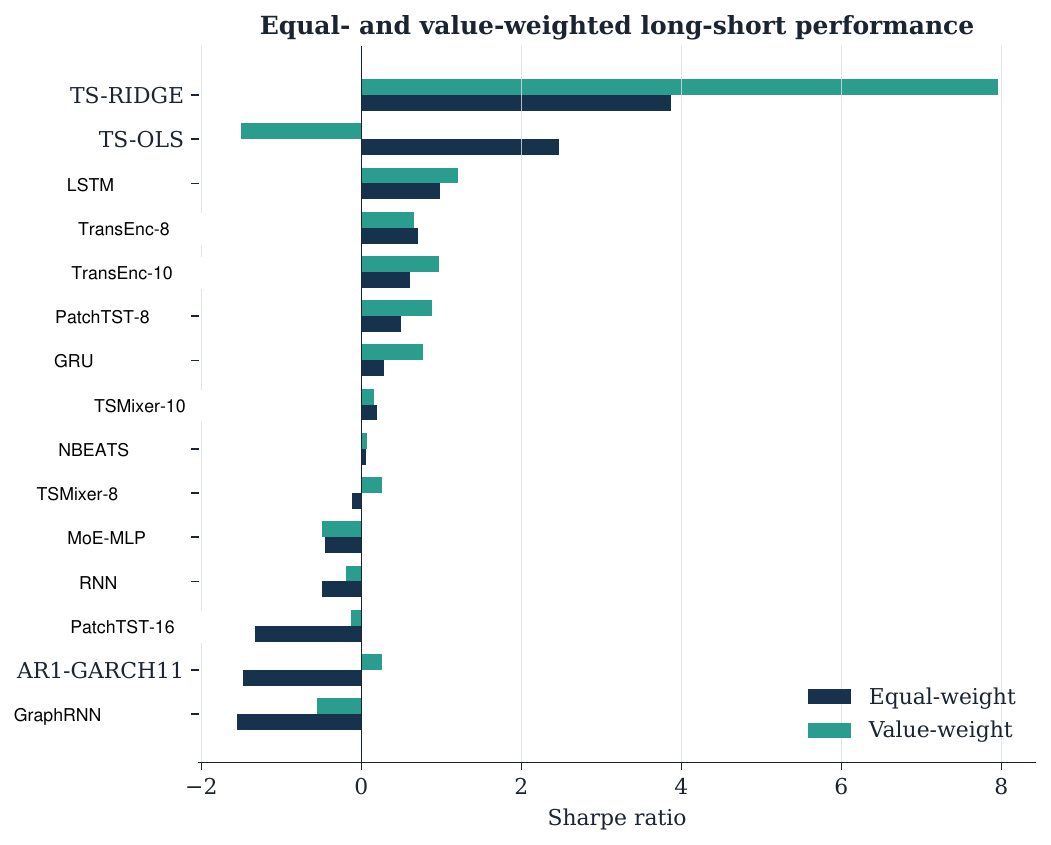}
    \caption{Equal-weighted and value-weighted Sharpe ratios for the fifteen retained models.}
    \label{fig:ewvw}
\end{figure}

Value-weighting strengthens TS-RIDGE and leaves several nonlinear architectures positive at lower magnitudes. TS-OLS turns negative under market-capitalization weights, separating shrinkage value from a plain linear score.

Turnover explains much of the gap between gross and net performance. High-Sharpe strategies with rapid entry, exit, or cross-side migration lose value once trading costs are charged. TS-RIDGE retains positive decile net performance at 20 bps despite substantial turnover, whereas LSTM, TransEnc-8, and TransEnc-10 survive through slower portfolio rotation; when the sort is widened from deciles to quintiles, concentration falls but the qualitative ranking among the leading models is preserved.

\begin{table}[!htbp]
\centering
\TableSetup
\caption{Model Confidence Set membership by return definition.}
\label{tab:mcs}
\FitTable{\input{table_mcs_model_confidence_set.tex}}
\end{table}

Table~\ref{tab:mcs} gives the Model Confidence Set at 90\% and 95\% confidence \citep{HansenLundeNason2011}. The gross MCS contains seven models, so most positive-Sharpe architectures cannot be statistically separated from the best performer. After 20 bps costs in the decile layer, the MCS narrows to TransEnc-8, LSTM, and TransEnc-10 at 90\% confidence, with PatchTST-8 and TSMixer-8 entering only at 95\%. TS-RIDGE exits the net MCS because high daily turnover erodes cost-adjusted performance, which leaves model choice to the preference analysis in Section~\ref{sec:smaa} rather than to a single inference table.

\section{Robust Multi-Criteria Model Selection}
\label{sec:smaa}
The preference-robust layer estimates the criteria distribution, rank-acceptability surface, and preference-space boundaries for the fifteen retained architectures. It implements Algorithm 1 with 10,000 moving-block resamples and 10,000 Dirichlet preference draws over the seven-criterion simplex.

The key reading convention is that SMAA rank-1 acceptability is not a return ranking. Each draw aggregates seven oriented criteria, so a model with negative gross Sharpe can still rank first in the subset of preference states that put enough weight on turnover, drawdown, or other implementation-friction criteria. PatchTST-16 is the useful warning case: despite a gross Sharpe of \(-1.33\) in the master benchmark, it receives a rank-1 acceptability share of 0.172 because its moderate realized turnover and lighter friction-related scores can dominate return when preferences move toward implementation. This behavior is exactly what SMAA is designed to reveal; it is not a claim that PatchTST-16 is the best return generator.

Two slices should also be kept separate. Figure~\ref{fig:smaaweightspace} scans a single low-turnover weight coordinate and finds a rank-1 leader switch near 0.47. Table~\ref{tab:smaastructuredpriors}, by contrast, uses structured Dirichlet priors over the full criterion vector; the net-performance prior still rewards decile net 20 bps Sharpe, where TS-RIDGE is strong, and therefore keeps TS-RIDGE as leader. The two outputs answer different preference questions and are not in conflict.

\begin{table}[!htbp]
\centering
\TableSetup
\caption{SMAA rank acceptability summary. Rank-1 and modal-rank quantities aggregate all seven oriented criteria and are not return-only rankings.}
\label{tab:smaaaccept}
\FitTable{\input{table_smaa_acceptability_top10.tex}}
\end{table}

The leading acceptability values in Table~\ref{tab:smaaaccept} show that TransEnc-8 and LSTM have the largest top-3 probabilities, both near 0.692, with TransEnc-10 next at 0.554. TS-RIDGE has a rank-1 probability of 0.199 and a top-3 probability of 0.293. PatchTST-16 has negative gross performance in the master benchmark, but it combines moderate realized turnover of 0.19 with lighter implementation-friction scores than several high-turnover losing architectures; when the preference draw puts enough mass on the low-turnover and drawdown-related coordinates, those oriented criteria can outweigh its return deficit and give it a 0.172 rank-1 share. This is a criterion-orientation consequence of the SMAA design: the table is an acceptability surface, and no model has rank-1 probability above 0.36.

\begin{table}[!htbp]
\centering
\TableSetup
\caption{SMAA rank-1 acceptability bands and modal ranks. The modal rank is the most frequently assigned SMAA rank under sampled criteria and preferences, and can be driven by non-return criteria such as turnover or drawdown.}
\label{tab:smaabands}
\FitTable{\input{table_smaa_rank1_bands.tex}}
\end{table}

Table~\ref{tab:smaabands} reports Monte Carlo bands around rank-1 acceptability and the modal rank for each leading architecture. TransEnc-8 has the largest rank-1 region, LSTM's modal rank is second, and TransEnc-10 concentrates in third, while TS-RIDGE has a narrower decision region than its gross Sharpe would suggest because turnover and drawdown receive positive weight in many preference states.

The finite-sample precision check is reported in the appendix. It uses 50,000 preference draws only as a convergence diagnostic, while the main SMAA design uses 10,000 moving-block criteria resamples and 10,000 Dirichlet preference draws.

The resampling layer quantifies sampling uncertainty in the realized criteria and preference uncertainty over the seven-criterion simplex. It does not quantify neural-network training randomness because the retained forecasts are produced by a fixed walk-forward training protocol; multi-seed refits are therefore a compute-dependent robustness extension rather than part of the present SMAA uncertainty measure.

\begin{figure}[!htbp]
    \centering
    \includegraphics[width=0.92\textwidth]{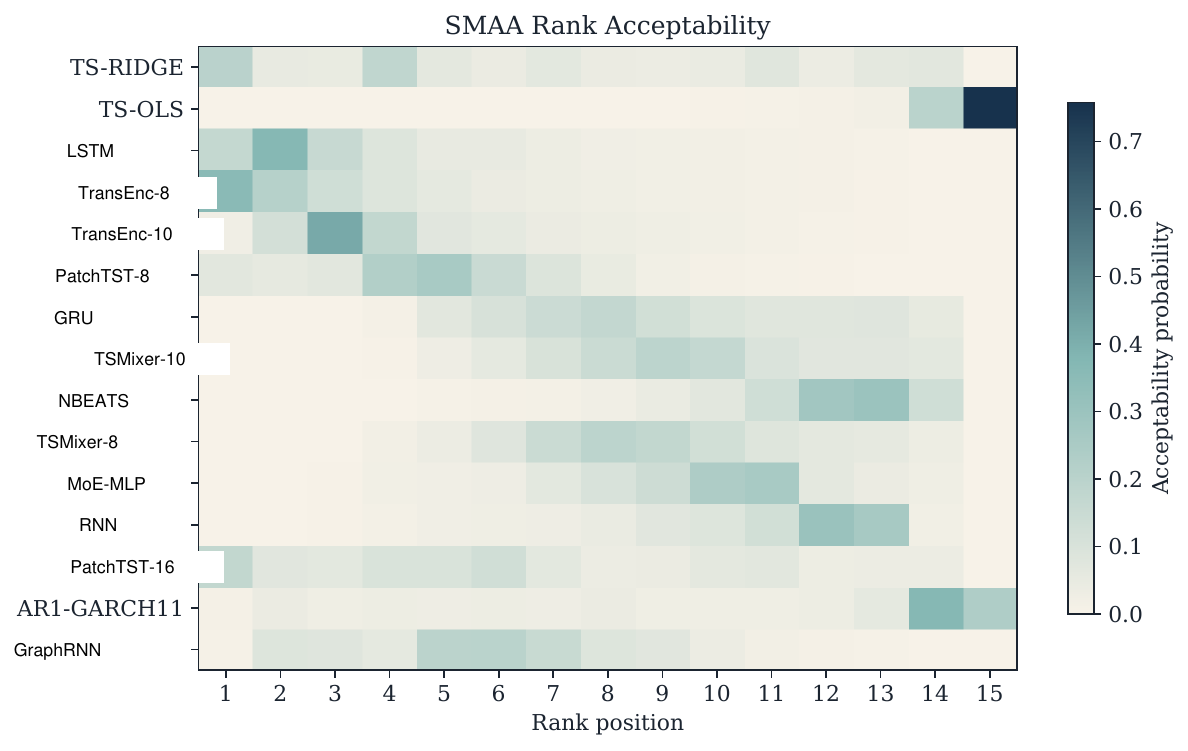}
    \caption{SMAA rank acceptability over 15 models and 15 possible ranks.}
    \label{fig:smaaheatmap}
\end{figure}

The full rank distribution in Figure~\ref{fig:smaaheatmap} reinforces the absence of a scalar winner. TransEnc-8 has the highest rank-1 share, LSTM spreads across ranks 1 through 3, TransEnc-10 concentrates in the third rank, and TS-RIDGE receives substantial rank-1 mass but less top-3 stability because turnover and cost criteria penalize it when their weights rise.

\begin{figure}[!htbp]
    \centering
    \includegraphics[width=0.82\textwidth]{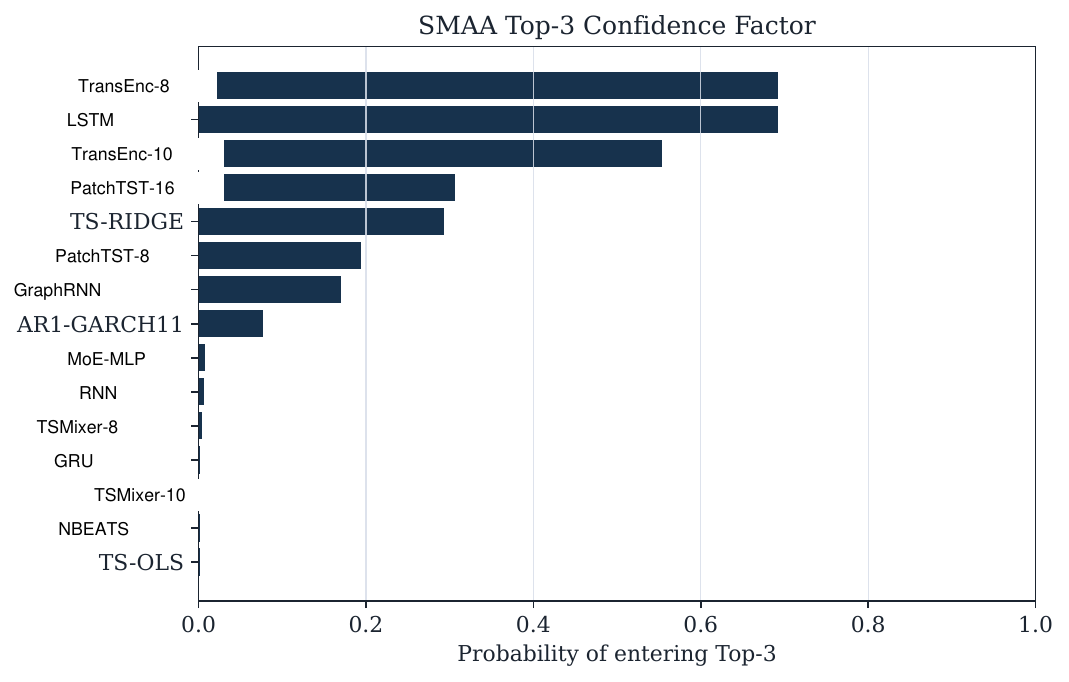}
    \caption{SMAA top-3 acceptability probabilities.}
    \label{fig:smaatop3}
\end{figure}

Top-3 acceptability concentrates around the lower-turnover transformer-encoder and recurrent configurations (Figure~\ref{fig:smaatop3}). TS-RIDGE remains attractive for return-oriented preferences and loses acceptability when trading-cost and drawdown criteria receive larger weights.

\begin{figure}[!htbp]
    \centering
    \includegraphics[width=0.70\textwidth]{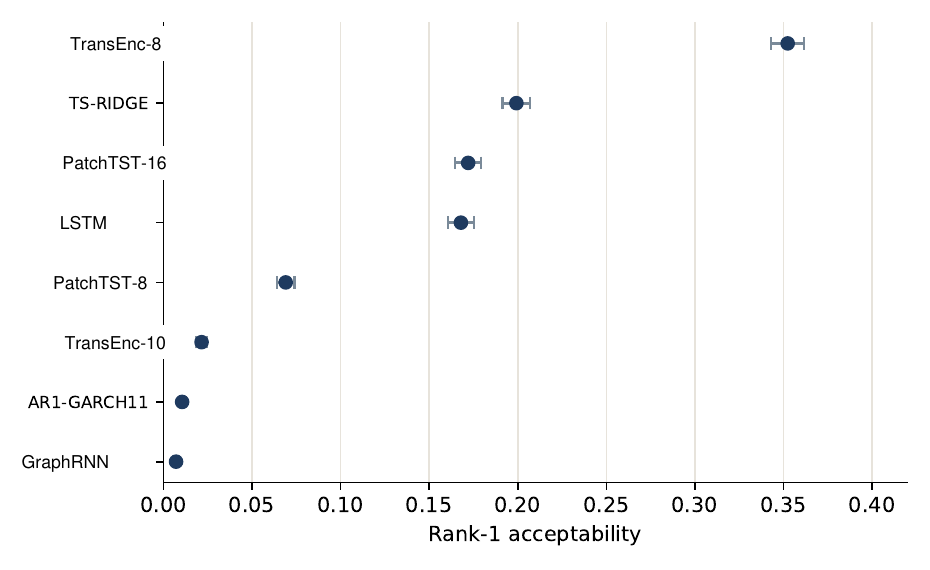}
    \caption{Rank-1 SMAA acceptability bands for the leading models.}
    \label{fig:smaabands}
\end{figure}

\begin{figure}[!htbp]
    \centering
    \includegraphics[width=0.86\textwidth]{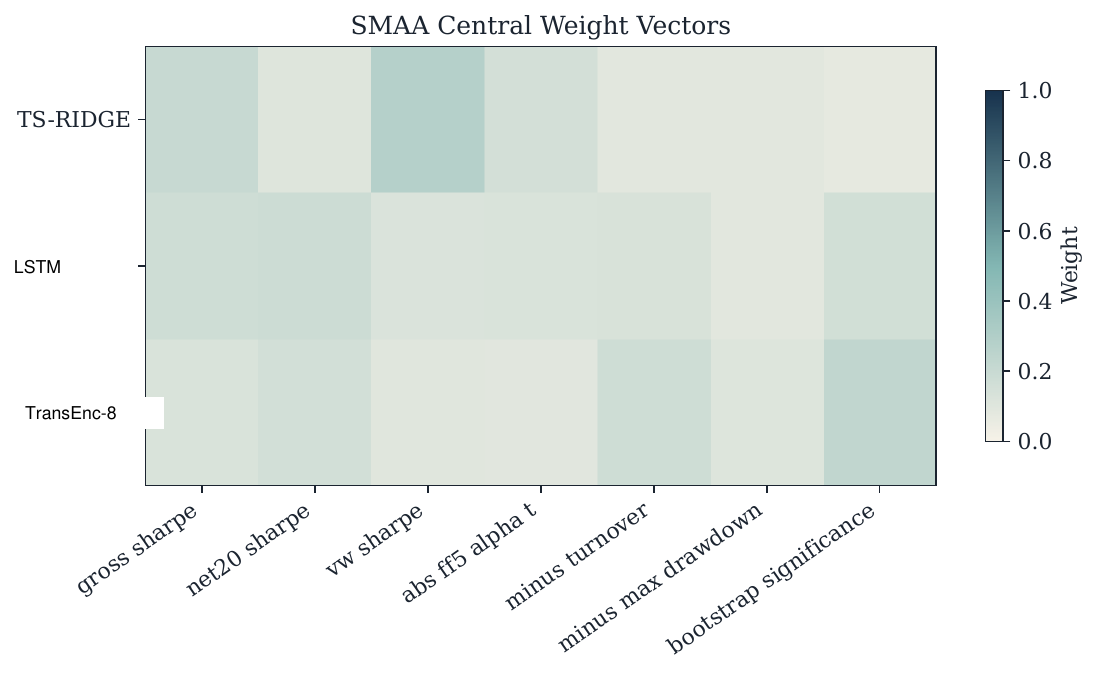}
    \caption{Central SMAA weight vectors for selected leading models.}
    \label{fig:smaacentral}
\end{figure}

Figure~\ref{fig:smaacentral} maps acceptability to investor profiles. A return-focused profile places high weight on gross and net Sharpe and selects TS-RIDGE. A cost-sensitive profile places high weight on turnover and net Sharpe and selects TransEnc-8 or LSTM. A capacity-focused profile places high weight on value-weighted Sharpe and drawdown and keeps TS-RIDGE in the candidate set. A statistical-rigor profile places high weight on bootstrap strength and factor alpha and favors TransEnc-8, LSTM, or TS-RIDGE depending on the alpha-versus-cost tradeoff.

\begin{figure}[!htbp]
    \centering
    \includegraphics[width=0.94\textwidth]{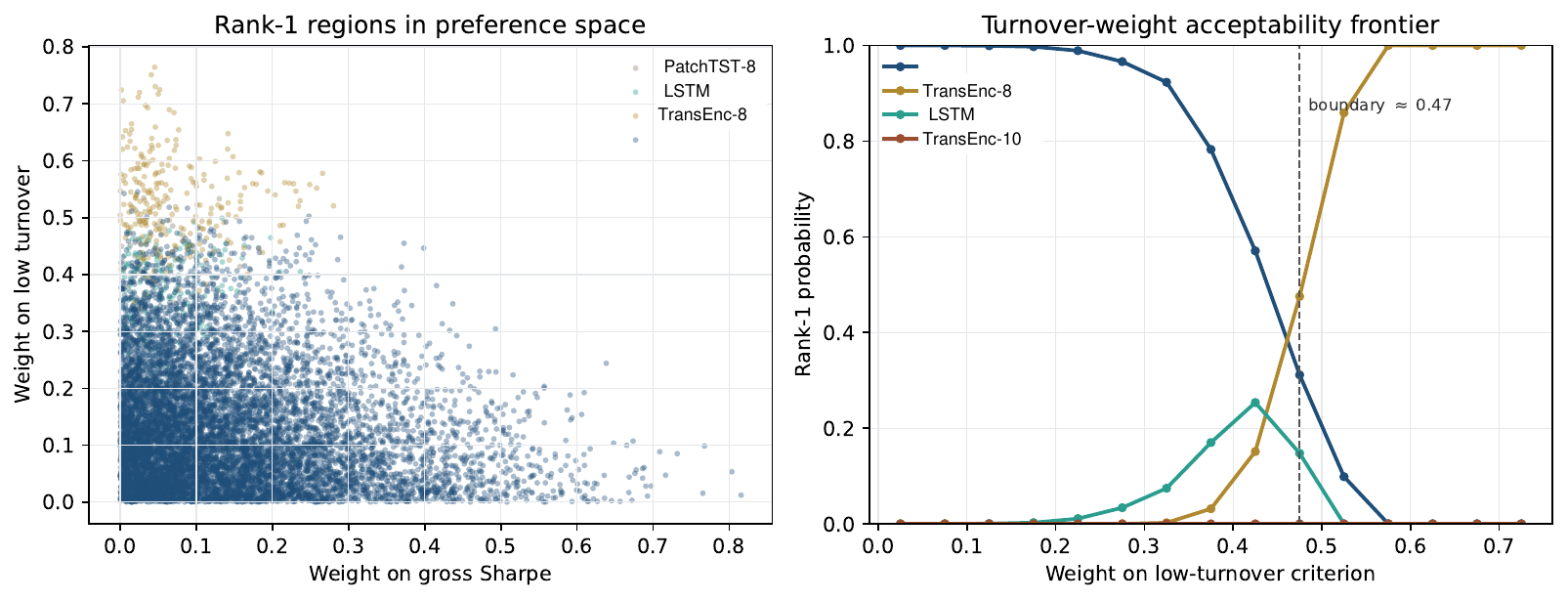}
    \caption{SMAA preference-space geometry and turnover-weight acceptability frontier.}
    \label{fig:smaaweightspace}
\end{figure}

Figure~\ref{fig:smaaweightspace} projects the sampled preference weights onto gross-Sharpe and low-turnover coordinates. TS-RIDGE occupies the region with high return weight and limited cost weight. TransEnc-8 gains rank-1 mass as the low-turnover criterion becomes more prominent. In this two-dimensional turnover-return projection, the empirical crossover occurs when the low-turnover weight reaches about 0.47 across the 10,000 Dirichlet weight draws. Other criteria in the full seven-dimensional simplex still move local boundaries, so the figure is a decision-threshold projection, not a global separating rule.

\begin{table}[!htbp]
\centering
\TableSetup
\caption{SMAA acceptability under structured preference priors. The net-performance prior is not a pure turnover-aversion prior: it still rewards decile net 20 bps Sharpe, which keeps TS-RIDGE as the rank-1 leader while increasing the secondary acceptability of lower-turnover sequence models.}
\label{tab:smaastructuredpriors}
\FitTable{\input{table_smaa_structured_priors.tex}}
\vspace{2pt}
\begin{minipage}{0.96\textwidth}
\footnotesize\emph{Note.} Selected leading models only; probabilities are computed over the full 15-model universe, so displayed rows need not sum to one.
\end{minipage}
\end{table}

Table~\ref{tab:smaastructuredpriors} replaces the uniform Dirichlet prior with return-oriented, net-performance-oriented, and statistical-rigor priors. TS-RIDGE remains the prior leader under all three structured weights. The net-performance prior is not a pure turnover-aversion prior; it still rewards decile net 20 bps Sharpe, where TS-RIDGE is strong, and therefore keeps the high-turnover shrinkage model as leader while raising TransEnc-8 and LSTM as secondary low-turnover alternatives. Under a uniform preference prior the index sharpens the lead of TransEnc-8; its decision content is clearest under return-oriented preferences, where the SMAA-preferred TS-RIDGE is overturned once the same scores pass through the constrained portfolio.

\begin{table}[!htbp]
\centering
\TableSetup
\caption{Annualized Sharpe ratios after one-way turnover costs across transaction cost levels. Costs are applied to the common-window daily return series using the same drift-adjusted one-way turnover reconstruction as Table~\ref{tab:masterbenchmark}.}
\label{tab:costsensitivity}
\FitTable{\input{table_cost_sensitivity_sharpe.tex}}
\end{table}

\begin{figure}[!htbp]
    \centering
    \includegraphics[width=0.72\textwidth]{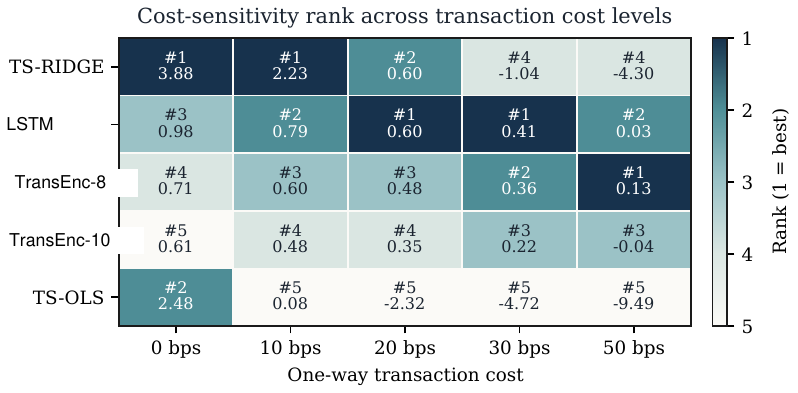}
    \caption{Ranking sensitivity across transaction cost levels.}
    \label{fig:costsensitivity}
\end{figure}

Table~\ref{tab:costsensitivity} and Figure~\ref{fig:costsensitivity} show the mechanism behind the SMAA shift. The Sharpe ratios are annualized after deducting one-way turnover costs at each stated basis-point level from the common-window return series. TS-RIDGE leads at zero and 10 bps, TS-RIDGE and LSTM are essentially tied at the rounded 20 bps decile net-Sharpe value, LSTM leads at 30 bps, and TransEnc-8 leads at 50 bps; as the cost schedule tightens, cost-sensitive preferences move acceptability toward architectures whose rankings rotate more slowly. Figure~\ref{fig:costsensitivity} uses the unrounded ordering when ranks differ inside a rounded tie.

\begin{table}[!htbp]
\centering
\TableSetup
\caption{Regime-conditional SMAA rank-1 acceptability.}
\label{tab:smaaregime}
\FitTable{\input{table_smaa_regime_rank1.tex}}
\end{table}

\begin{figure}[!htbp]
    \centering
    \includegraphics[width=0.94\textwidth]{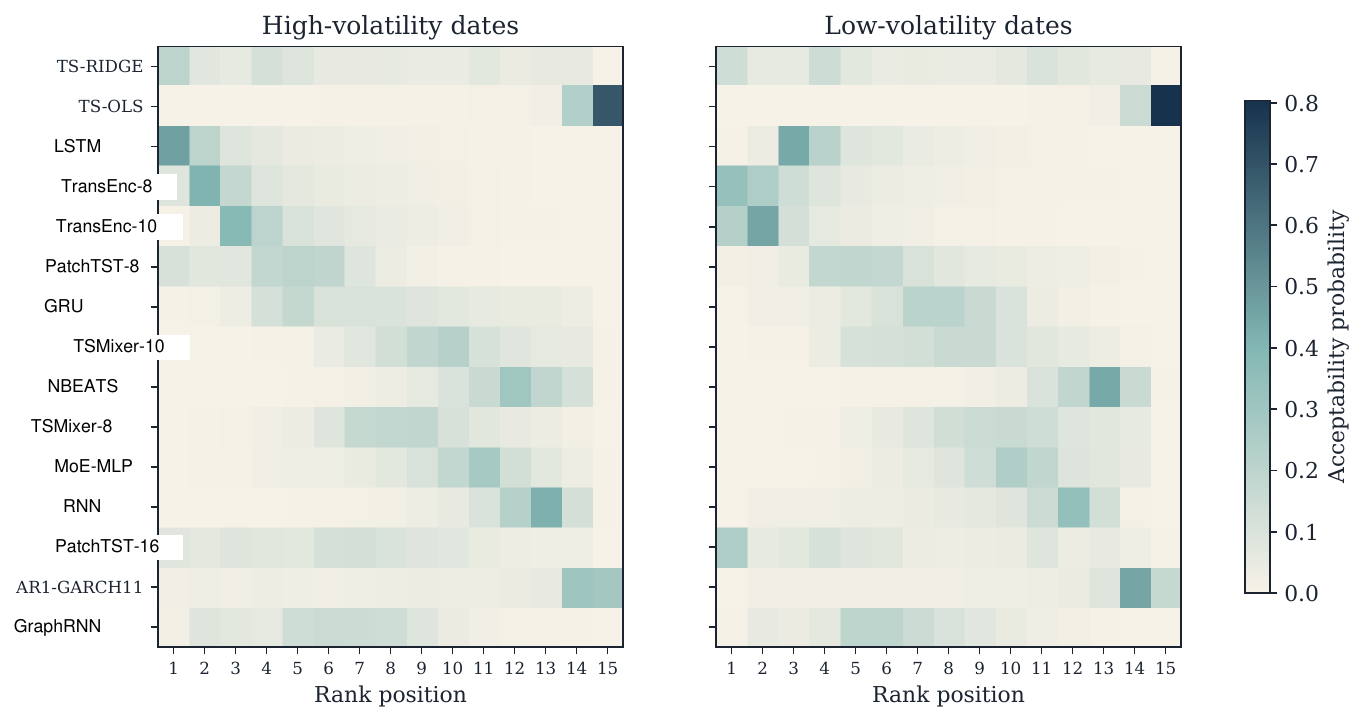}
    \caption{SMAA rank acceptability by market-volatility state.}
    \label{fig:smaaregime}
\end{figure}

Table~\ref{tab:smaaregime} and Figure~\ref{fig:smaaregime} split the acceptability analysis by market-volatility state. LSTM has rank-1 probability of 0.470 in high-volatility dates and 0.009 in low-volatility dates, while TransEnc-8 moves in the opposite direction, rising from 0.089 to 0.337. TransEnc-10 also gains in low-volatility dates, with rank-1 probability of 0.230 and top-3 probability of 0.811, and TS-RIDGE keeps positive rank-1 mass in both states, at 0.198 and 0.144. The leading architectures persist across states, but the preferred member shifts with volatility.

\begin{table}[!htbp]
\centering
\TableSetup
\caption{Sequential split-sample SMAA validation.}
\label{tab:smaaoos}
\FitTable{\input{table_smaa_oos_validation.tex}}
\end{table}

Table~\ref{tab:smaaoos} gives the sequential split-sample check. The first half of the evaluation period has Kendall tau of 0.62 and Spearman correlation of 0.80 with the second-half SMAA ranking; TransEnc-10, TransEnc-8, and LSTM remain near the top across the split, while TS-RIDGE retains a middle-rank region with stable top-3 probability. The acceptability ranking has measurable decision content, but it is not static, so updating the decision analysis with new market data is part of the model-selection rule.

\section{Constrained Portfolio Construction}
\label{sec:portfolio}
The common-window decile portfolio treats the model score as a sorting rule. The deployment-adjusted layer implements Problem~\ref{prob:daily-qp} and Algorithm~2. The same score enters a constrained daily portfolio map, and realized next-day returns determine deployment loss. Proposition~\ref{thm:regret-lipschitz} gives the ex-ante sensitivity channel behind the regret tables, including the additional downstream-state term that appears when previous holdings differ across model paths, while Proposition~\ref{thm:robust-shrinkage} interprets score ambiguity as an allocation-level shrinkage penalty.

\begin{table}[!htbp]
\centering
\TableSetup
\caption{Constrained portfolio optimization under wider implementation designs.}
\label{tab:qpsummary}
\FitTable{\input{table_qp_spec_spread.tex}}
\end{table}

Table~\ref{tab:qpsummary} gives the baseline design, a stronger turnover-penalty design, and a 50 bps cost stress, with annual returns reported as decimal annualized returns. The baseline and higher-turnover-penalty designs use the one-way 20 bps cost formula in Problem~\ref{prob:daily-qp}; the stress column replaces 0.002 by 0.005. The low-risk variant is merged into the baseline because the risk penalty does not bind over this parameter range. All optimized portfolios have negative constrained-QP net Sharpe after 20 bps costs: TransEnc-8 is the least negative baseline design, with net Sharpe of $-0.76$, followed by TransEnc-10 at $-1.18$ and LSTM at $-1.25$, while TS-RIDGE falls to $-2.37$ because its high-turnover signal remains costly inside the constrained daily portfolio.

\begin{table}[!htbp]
\centering
\TableSetup
\caption{Robust valuation frontier for constrained allocations under ellipsoidal score ambiguity.}
\label{tab:drofrontier}
\FitTable{\input{table_v8_dro_frontier.tex}}
\end{table}

\begin{figure}[!htbp]
    \centering
    \includegraphics[width=0.70\textwidth]{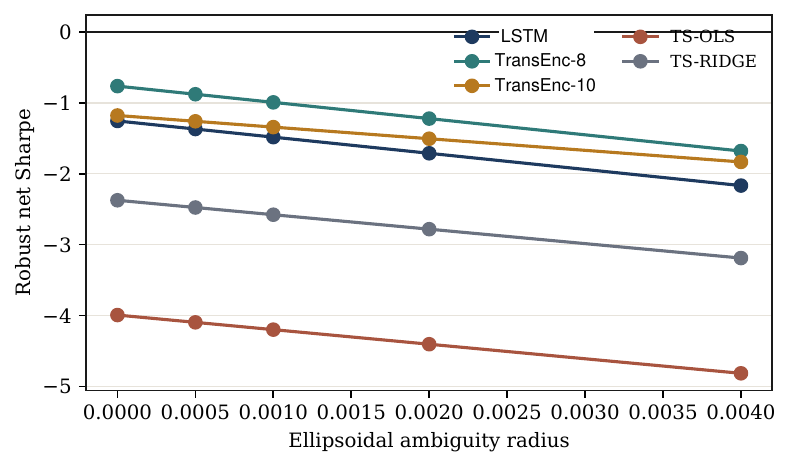}
    \caption{Robust net Sharpe under ellipsoidal score ambiguity.}
    \label{fig:drofrontier}
\end{figure}

Table~\ref{tab:drofrontier} and Figure~\ref{fig:drofrontier} value the constrained allocations under the support-function penalty in Proposition~\ref{thm:robust-shrinkage}. The perturbation radius lowers robust net Sharpe monotonically for every model. TransEnc-8 remains the least negative allocation across the radius grid. The larger penalty for high-norm allocations shows the robust-optimization channel: score ambiguity charges concentration even when leverage is fixed. Table~\ref{tab:drofrontier} reports the worst-case robust valuation of Proposition~\ref{thm:robust-shrinkage}, in which the radius is charged against the fixed nominal weights. Re-solving the robust program day by day instead---letting the optimizer de-concentrate to absorb the penalty---leaves realized net Sharpe essentially unchanged over the same radius grid: TransEnc-8 stays at \(-0.76\) and remains the least-negative allocation at every \(\rho\), versus \(-1.68\) under the ex-post valuation at \(\rho=0.004\). The full re-optimized frontier is reported in the appendix. The day-by-day robust optimization therefore confirms that the deployment-boundary ranking is not an artifact of ex-post valuation.

\begin{table}[!htbp]
\centering
\TableSetup
\caption{Reference benchmarks for the constrained-portfolio layer. The final column is the gap versus the best reference row in this table, not the model-set deployment regret in Definition~3.}
\label{tab:naivereference}
\FitTable{\input{table_v8_naive_qp_reference.tex}}
\end{table}

Table~\ref{tab:naivereference} adds simple reference lines to the constrained-portfolio results. The market factor and no-trade cash are not architecture-selection rules; they anchor the scale of deployment loss. The final column is a gap versus the best row in this reference table, which is the market factor, and is therefore separate from the model-set deployment regret in Table~\ref{tab:decisionregret}. Since the forecast-driven constrained portfolios all have negative constrained-QP net Sharpe at 20 bps, the optimized layer ranks implementation losses. It does not by itself establish a profitable trading rule.

\begin{figure}[!htbp]
    \centering
    \includegraphics[width=0.86\textwidth]{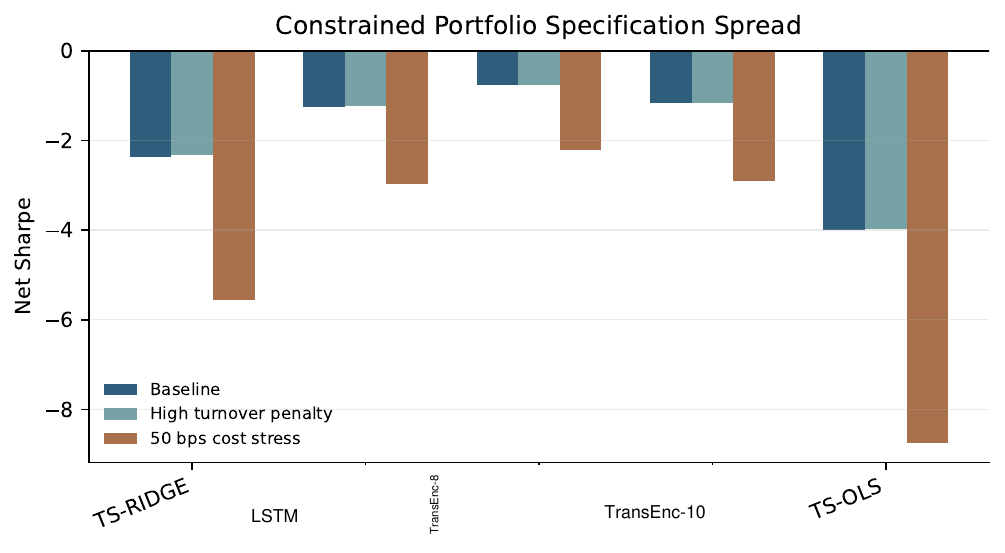}
    \caption{Net Sharpe ratios under wider constrained-portfolio designs.}
    \label{fig:qpspecspread}
\end{figure}

Figure~\ref{fig:qpspecspread} summarizes the optimization-level cost pattern. The baseline and stronger turnover-penalty designs preserve the ordering, whereas the 50 bps stress creates visible dispersion and lowers high-turnover strategies. The constrained portfolio layer is not a decile capacity frontier: it uses continuous dollar weights, leverage, dollar-volume caps, beta and industry constraints, and an explicit turnover-control objective, yet strict next-day accounting still leaves the daily optimized portfolios below zero after costs.

SMAA penalizes TS-RIDGE when turnover receives criterion weight, and the constrained portfolio confirms that penalty under next-day realized returns. Across the transaction-cost grid, the optimized ranking is stable: TransEnc-8, TransEnc-10, LSTM, TS-RIDGE, and TS-OLS. The association reported in Figure~\ref{fig:predictoptgap} is computed between the raw SMAA expected-rank ordering restricted to the promoted five and the constrained-QP ranking (TransEnc-8, TransEnc-10, LSTM, TS-RIDGE, TS-OLS). This pairing gives Kendall's $\tau=0.80$ and Spearman's $\rho=0.90$ at every cost level from 0 to 50 bps; the raw top-3 ordering gives the same rounded association, while the raw rank-1 restriction does not.

\begin{figure}[!htbp]
    \centering
    \includegraphics[width=0.66\textwidth]{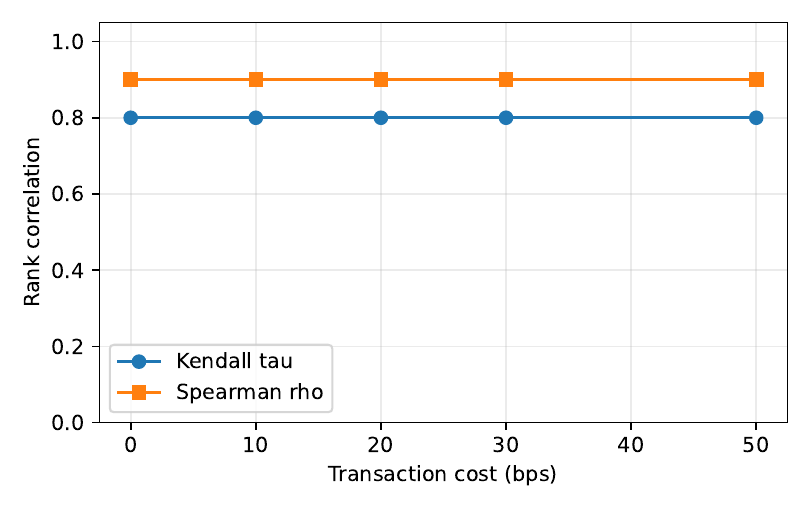}
    \caption{Predict-then-optimize rank association between raw SMAA expected ranks and constrained-QP ranks across transaction cost levels.}
    \label{fig:predictoptgap}
\end{figure}

The rank-association curve in Figure~\ref{fig:predictoptgap} is the empirical counterpart of Proposition~\ref{thm:turnover-alignment}. Criteria-based turnover weights and portfolio turnover control point to the same low-turnover leader once the decision layer is imposed. Across signal universes, the role pattern remains conditional: shrinkage wins in the full structured decile universe, and constrained daily deployment favors TransEnc-8.

\begin{table}[!htbp]
\centering
\TableSetup
\caption{Deployment regret under constrained-portfolio designs.}
\label{tab:decisionregret}
\FitTable{\input{table_decision_regret_by_qp_design.tex}}
\end{table}

\begin{figure}[!htbp]
    \centering
    \includegraphics[width=0.78\textwidth]{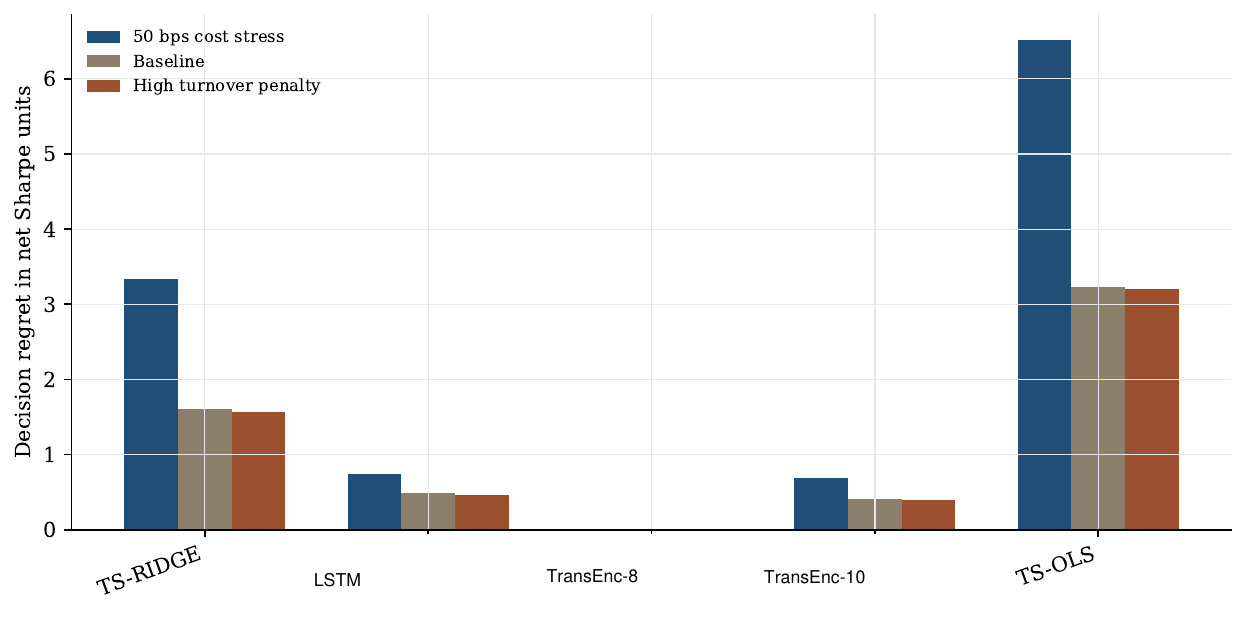}
    \caption{Deployment regret in net Sharpe units under constrained-portfolio designs.}
    \label{fig:decisionregret}
\end{figure}

Table~\ref{tab:decisionregret} converts the optimized ranking into decision cost. Relative to TransEnc-8, the baseline regret is 0.41 Sharpe units for TransEnc-10, 0.49 for LSTM, 1.61 for TS-RIDGE, and 3.23 for TS-OLS. The raw SMAA rank-1 rule and the QP decision rule both select TransEnc-8, so the cross-framework regret is zero under the baseline design. Proposition~\ref{thm:regret-lipschitz} explains why models with similar score vectors and similar downstream states have bounded decision-value gaps, while Figure~\ref{fig:decisionregret} shows that the ordering persists under the high-turnover penalty and the 50 bps stress.

\begin{table}[!htbp]
\centering
\TableSetup
\caption{Raw and optimized SMAA rank acceptability for promoted models. The raw columns inherit the 15-model decile SMAA and are restricted here to the promoted subset; the optimized columns rerun SMAA on the five-model constrained-portfolio universe.}
\label{tab:optimizedsmaa}
\FitTable{\input{table_optimized_smaa_gap.tex}}
\end{table}

\begin{figure}[!htbp]
    \centering
    \includegraphics[width=0.86\textwidth]{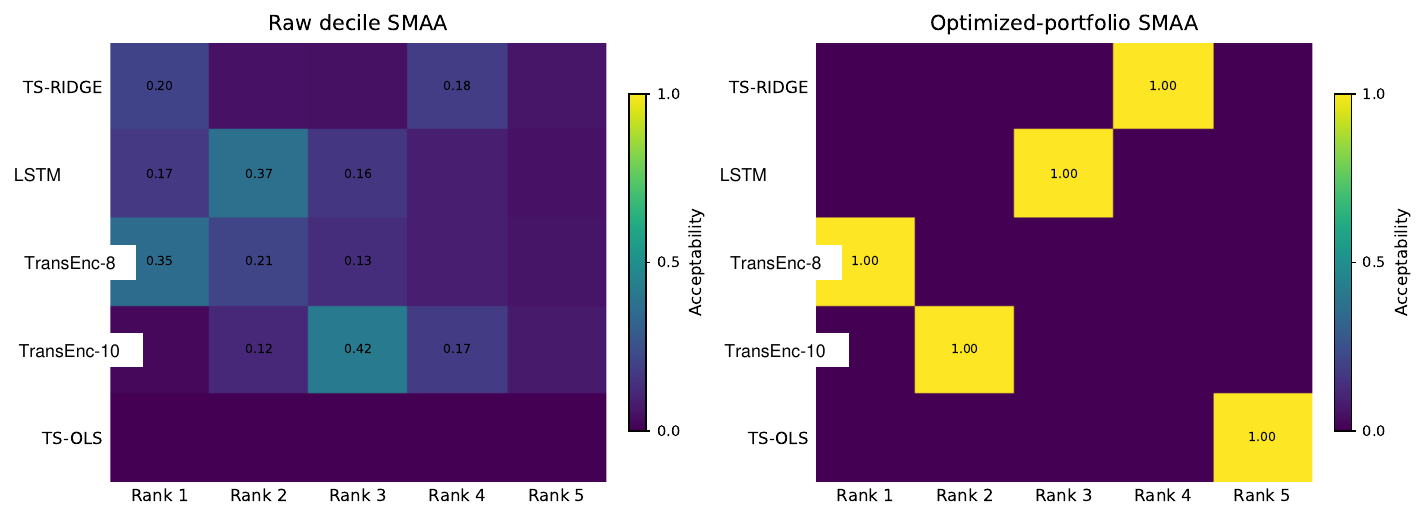}
    \caption{Raw 15-model decile SMAA, restricted to promoted models, and optimized five-model portfolio SMAA.}
    \label{fig:optimizedsmaa}
\end{figure}

Table~\ref{tab:optimizedsmaa} applies SMAA to optimized portfolio outcomes. Its raw columns keep the original 15-model decile-SMAA universe and display only the promoted models, whereas the optimized columns recompute acceptability in the five-model constrained-portfolio universe. TransEnc-8 rank-1 acceptability rises from 0.352 in the raw decile evaluation to 1.000 after optimization. TS-RIDGE falls from 0.199 to zero. The constrained portfolio layer removes most preference sensitivity in favor of the low-turnover attention model, but the optimized comparison is a relative loss ranking because all net Sharpe ratios are negative.

\begin{table}[!htbp]
\centering
\TableSetup
\caption{Deployment-adjusted acceptability after weighting SMAA mass by constrained-portfolio regret.}
\label{tab:decisionawareaccept}
\FitTable{\input{table_deployment_adjusted_acceptability.tex}}
\end{table}

\begin{figure}[!htbp]
    \centering
    \includegraphics[width=0.72\textwidth]{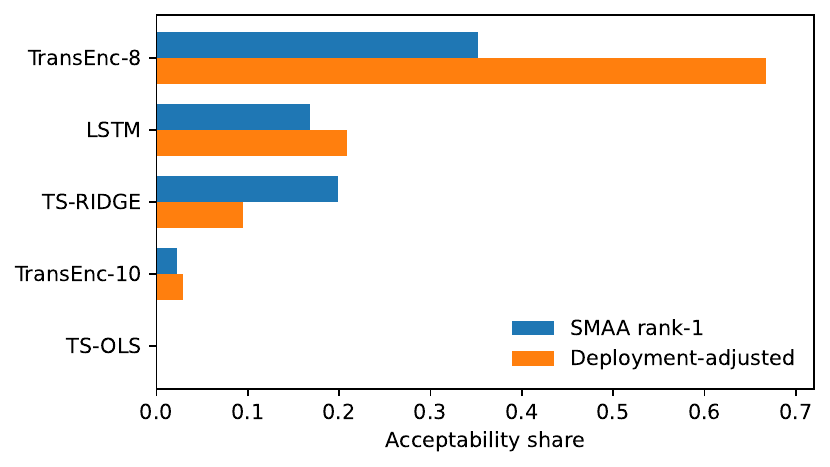}
    \caption{Raw and deployment-adjusted rank-1 acceptability for promoted models.}
    \label{fig:decisionawareaccept}
\end{figure}

The deployment-adjusted index in Table~\ref{tab:decisionawareaccept} discounts each model's SMAA rank-1 mass by downstream QP regret using Definition~4 and Proposition~\ref{prop:da-entropic}. In the baseline constrained-QP calculation, the implemented discount is \(\exp\{-\mathcal R_m^d/\tau_d\}\) with \(\tau_d\approx1.170\), equivalently \(\gamma_d=1/\tau_d\). TransEnc-8's deployment-adjusted rank-1 share is 0.667, LSTM follows at 0.209, and TS-RIDGE falls to 0.095 because its criteria-level rank-1 region carries a larger constrained-portfolio regret. A sensitivity check using \(c\in\{0.25,0.5,1,2,4\}\) in \(\exp\{-c\mathcal R_m^d/\tau_d\}\) preserves the ordering TransEnc-8 \(>\) LSTM \(>\) TS-RIDGE. This entropically regularized acceptability measure is the paper's link between preference uncertainty and realized deployment loss.

\section{Signal Mechanisms and Feature Boundaries}
\label{sec:mechanisms}
The input space determines the value of each architecture. Table~\ref{tab:d13main} reports the five promoted models across F1, F2, and F3 on gross long--short returns over the confirmatory window. The F3 rows use the same full-universe forecast files as the confirmatory headline benchmark, so the shared F3 cells match Table~\ref{tab:confirmatorywindowbenchmark} exactly for TS-RIDGE, TS-OLS, LSTM, TransEnc-8, and TransEnc-10. LSTM is strongest when the input is dominated by price-path information, whereas TS-RIDGE reasserts itself when the full structured signal stack becomes available; the intermediate universe is where complementarity becomes most informative, because latent-state nonlinear structure helps with recent price dynamics while shrinkage becomes valuable once the predictor set expands.

\begin{table}[!htbp]
\centering
\TableSetup
\caption{Promoted-model gross long--short results across the three nested signal universes, 2021-01-04 to 2024-12-30. F3 uses the full-universe forecast files from the headline benchmark; returns are long-high minus short-low gross returns, and RankIC is the daily cross-sectional Spearman correlation between the same raw score and next-day returns.}
\label{tab:d13main}
\FitTable{\input{table_direction13_main.tex}}
\end{table}

RankIC in Table~\ref{tab:d13main} is the time-series average of daily cross-sectional Spearman correlations between raw model scores and next-day returns. The long--short Sharpe ratio uses the same score direction but only the extreme deciles, and an alignment check confirms that both quantities use next-day returns rather than same-date returns. The important case is TS-RIDGE in F3: its average RankIC is essentially zero, but its gross Sharpe is 4.18 because the predictive content concentrates in the extreme deciles, so a near-zero full-rank correlation can coexist with a strong top-minus-bottom spread. The same logic motivates the decile-return matrices and sign-direction checks for rows whose RankIC and gross Sharpe have opposite signs.

\begin{figure}[!htbp]
    \centering
    \includegraphics[width=0.86\textwidth]{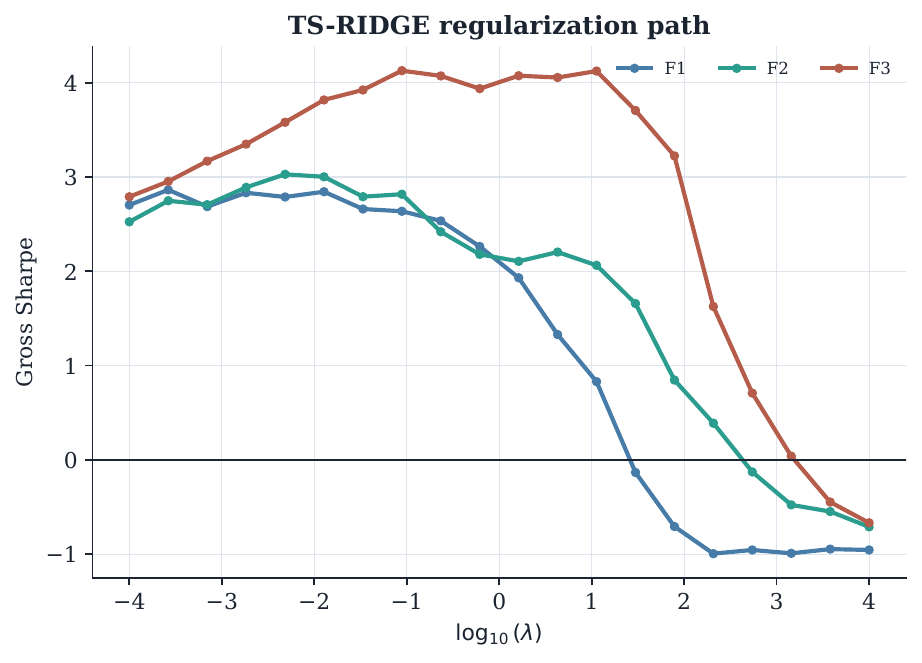}
    \caption{TS-RIDGE regularization path across nested feature universes.}
    \label{fig:ridgelambda}
\end{figure}

The ridge path makes the shrinkage mechanism explicit: as the feature universe expands, Gross-Sharpe optima move toward stronger regularization, and the best full-signal penalty becomes larger than the best price-path penalty. The full structured signal set contains useful information, but that information becomes valuable only when coefficient discipline prevents overreaction to unstable high-dimensional variation.

\begin{figure}[!htbp]
    \centering
    \includegraphics[width=0.88\textwidth]{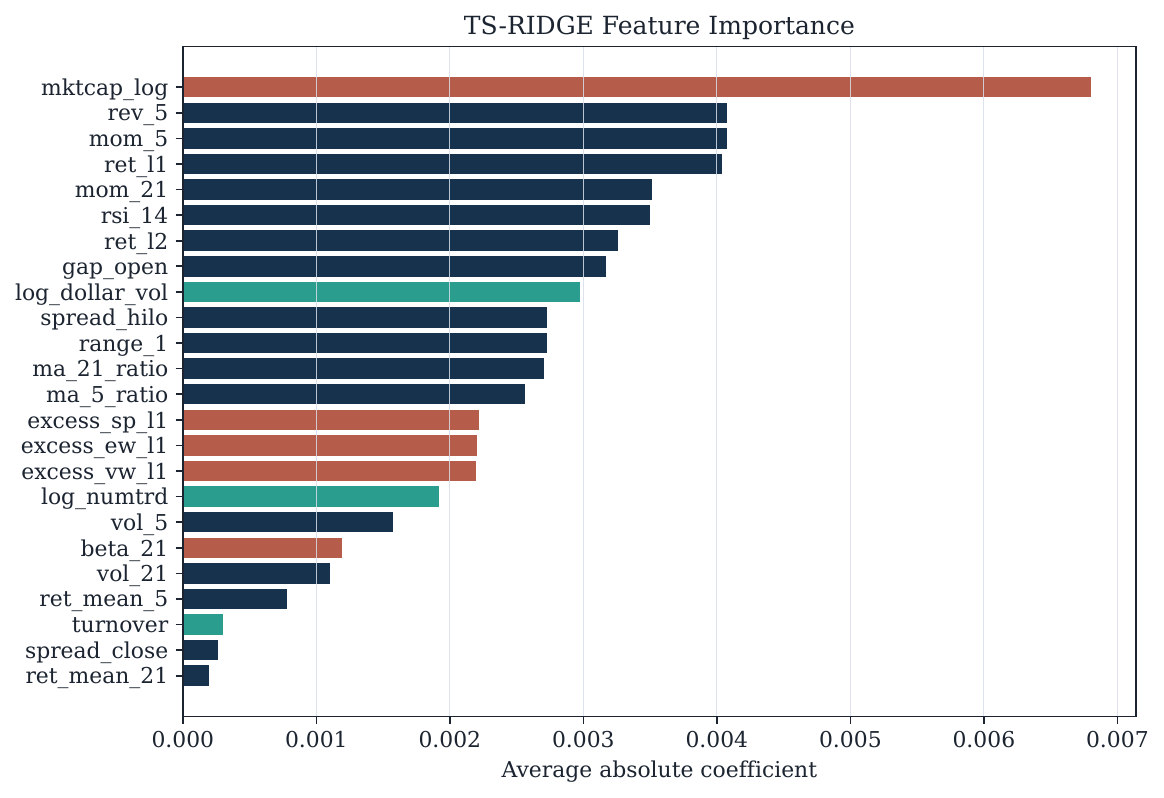}
    \caption{Average absolute TS-RIDGE coefficient by feature.}
    \label{fig:ridgeimportance}
\end{figure}

\begin{table}[!htbp]
\centering
\TableSetup
\caption{Largest TS-RIDGE coefficients in the full signal universe.}
\label{tab:ridgeimportance}
\FitTable{\input{table_ridge_f3_feature_importance_top12.tex}}
\end{table}

Figure~\ref{fig:ridgeimportance} and Table~\ref{tab:ridgeimportance} give average absolute Ridge coefficients across the 1,197 evaluation dates. The largest single coefficient belongs to the market-capitalization proxy, which is available only in the F3 universe, and this feature profile explains the F1-to-F3 performance jump for TS-RIDGE because the full signal stack gives Ridge access to cross-sectional size variation absent from the price-path universe. Short-horizon reversal and momentum features provide the next-largest coefficients; the size/beta block supplies the largest single effect, while the price-path block supplies a broader set of active predictors.

\begin{figure}[!htbp]
    \centering
    \includegraphics[width=0.90\textwidth]{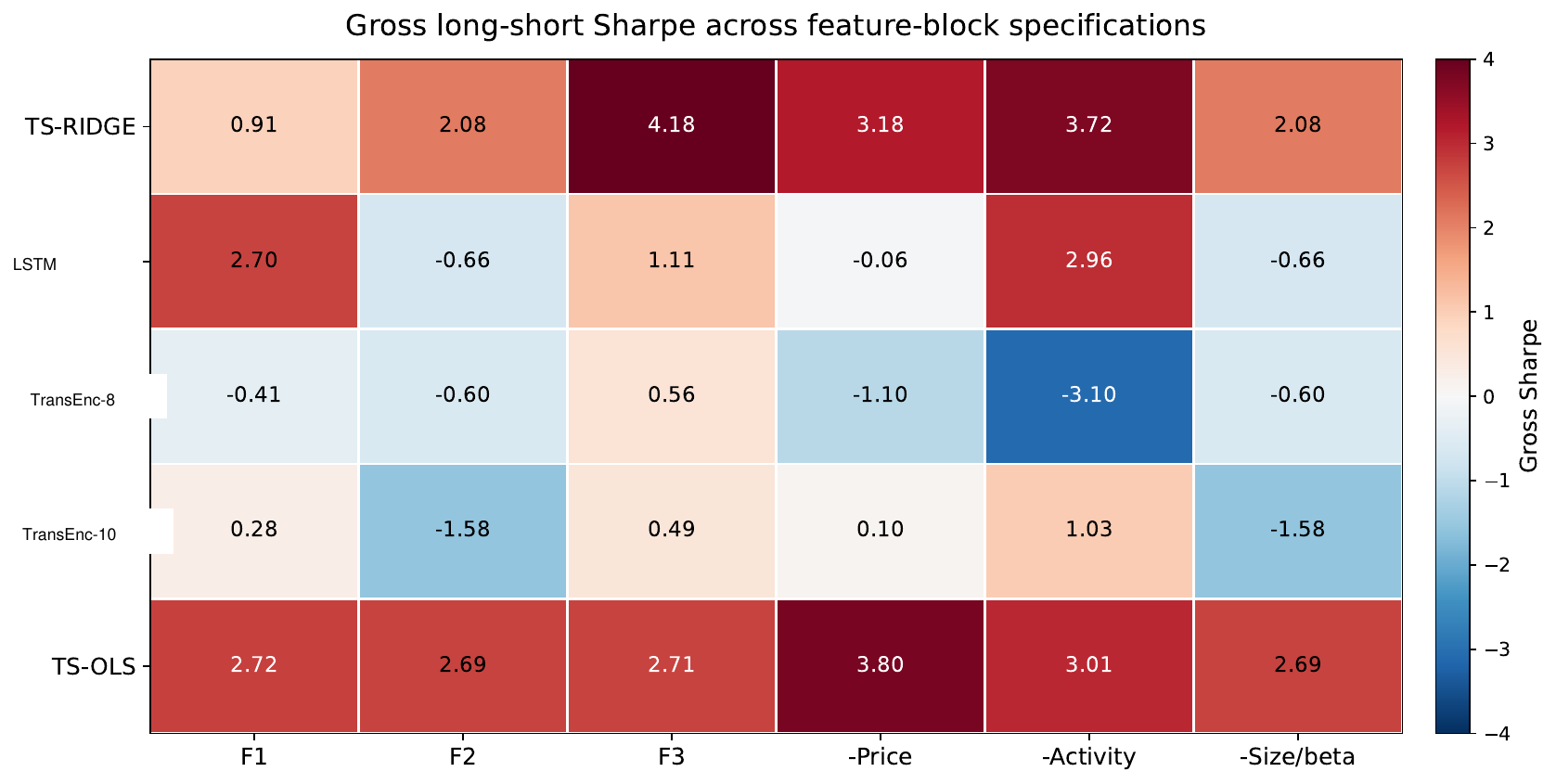}
    \caption{Feature-block ablation heatmap for the five promoted models under the common forecast protocol.}
    \label{fig:ablation}
\end{figure}

The feature-block ablation reinforces the same mechanism. Because F2 is F3 without the size, beta, and market-relative block, removing the price block weakens the nonlinear price-path channel and removing activity variables tests whether trading intensity carries the ranking. Restoring the full structured stack increases the value of shrinkage, and the plain transformer-encoder and recurrent configurations remain sensitive to feature richness and sample length. The shared F3 cells were checked against the confirmatory benchmark and match for the five promoted models.

Complementarity is evaluated after signal geometry is visible. Rank correlations separate close substitutes from differentiated rankings with weak standalone economics, and a useful blend must therefore combine signal distinctiveness with positive economic evidence rather than low correlation alone.

\begin{figure}[!htbp]
    \centering
    \includegraphics[width=0.88\textwidth]{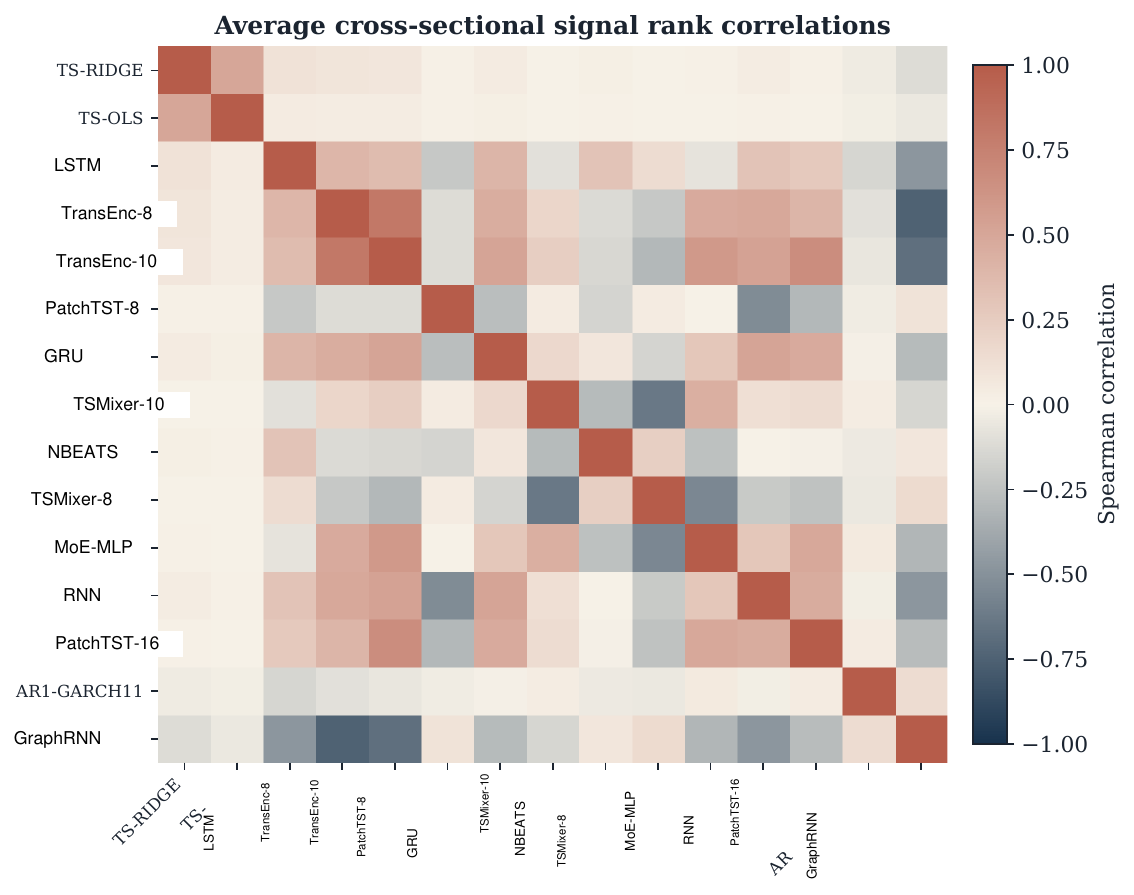}
    \caption{Average daily cross-sectional Spearman rank correlations among retained model prediction signals.}
    \label{fig:signalcorr}
\end{figure}

\begin{figure}[!htbp]
    \centering
    \includegraphics[width=0.86\textwidth]{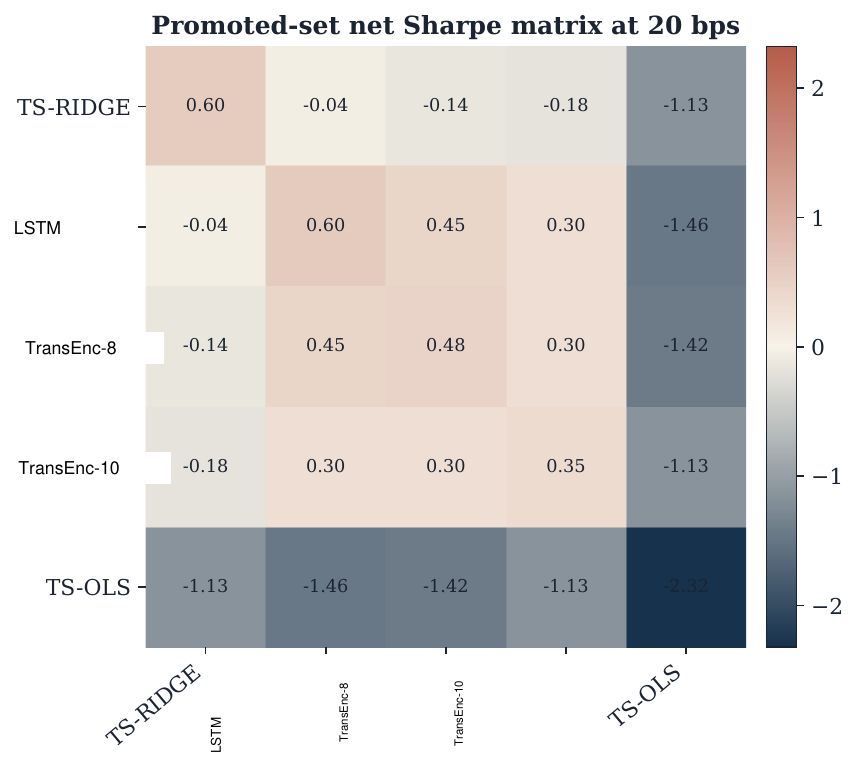}
    \caption{Promoted-set pairwise rolling-combination net Sharpe matrix.}
    \label{fig:combomatrix}
\end{figure}

The TS-RIDGE--LSTM pair is the central interaction. The common-window equal blend over all F3 features has low signal correlation and positive gross Sharpe--decile net Sharpe at 20 bps is close to zero--while the confirmatory blend moves to the intermediate F2 universe, where price-path and activity information leave more room for Ridge and LSTM to contribute orthogonal rankings. In the diagnostic tables, the internal object codes have been replaced by descriptive labels: F1 LSTM long--short, F2 TS-RIDGE long--short, F2 LSTM long--short, F2 Ridge--LSTM combo, and F3 TS-RIDGE long--short. TS-RIDGE supplies the stable full-signal linear anchor, and LSTM supplies the price-sensitive nonlinear contrast.

\begin{table}[!htbp]
\centering
\TableSetup
\caption{Risk and exposure diagnostics for promoted models and central portfolios.}
\label{tab:riskdiagnostics}
\input{table_risk_diagnostics_dual_panel.tex}
\end{table}

The market-proxy and holdings-exposure diagnostics reduce the chance that the main portfolios are disguised beta, size, liquidity, price, or industry exposures. The market-proxy rows use the same common-window return series as the five-factor regressions in Table~\ref{tab:masterbenchmark}. TS-RIDGE and TS-OLS have large promoted-model alphas, LSTM is statistically positive, and the F2 TS-RIDGE--LSTM combination remains the clearest validation-portfolio alpha.

\section{Deployment Boundaries}
\label{sec:capacity}
Daily stock sorting draws strength from universe breadth, so restricting capacity changes the ranking. Figure~\ref{fig:capacityfrontier} traces the capacity frontier for the promoted models, plotting decile net Sharpe at 20 bps from the full stock universe to the largest 100 stocks.

\begin{figure}[!htbp]
    \centering
    \includegraphics[width=0.88\textwidth]{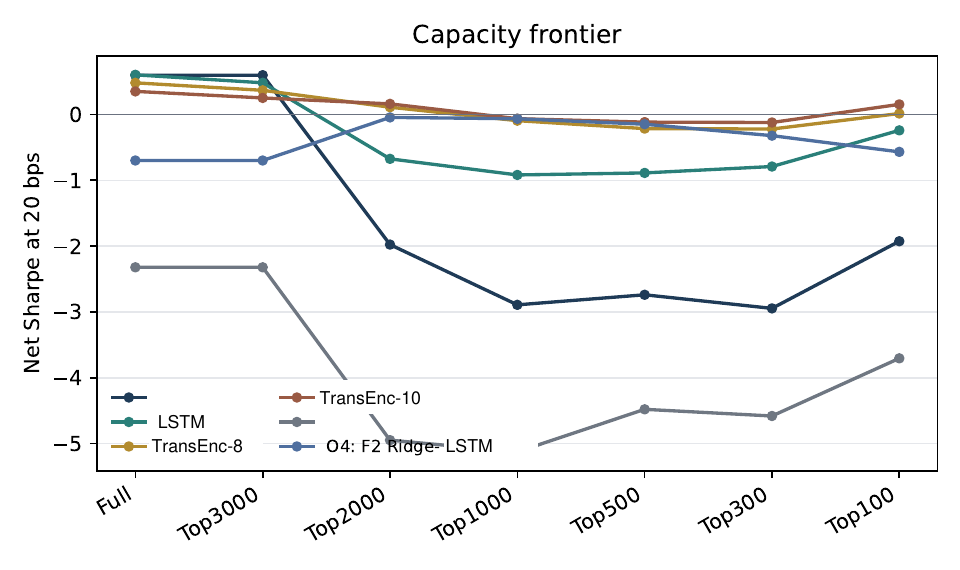}
    \caption{Capacity frontier for promoted models.}
    \label{fig:capacityfrontier}
\end{figure}

\begin{table}[!htbp]
\centering
\TableSetup
\caption{Capacity frontier: decile net Sharpe at 20 bps.}
\label{tab:capacityfrontier}
\FitTable{\input{table_capacity_frontier_net20.tex}}
\end{table}

Table~\ref{tab:capacityfrontier} gives the same frontier values, with the Full and Top3000 rows matching the common-window cost-adjusted benchmark. TS-RIDGE has net Sharpe of 0.60 in the broad universe but turns negative once the universe is restricted below the largest 2,000 names; LSTM is similar in the broad universe and also weakens under tighter capacity, whereas TransEnc-8 and TransEnc-10 have lower broad-universe Sharpe but lose less from large-cap concentration. The frontier shows that most daily alpha in this design comes from breadth, not from a large-cap core.

\begin{table}[!htbp]
\centering
\TableSetup
\caption{Positive net-performance configurations across implementation restrictions.}
\label{tab:deployablepositive}
\FitTable{\input{table_deployable_positive_configs.tex}}
\end{table}

Table~\ref{tab:deployablepositive} consolidates the settings in which decile net performance remains positive. The strongest entries come from low-price or small-stock buckets and from weekly rebalancing, while broad-universe TS-RIDGE and LSTM remain positive. Large-cap capacity restrictions weaken most daily decile portfolios, so deployability concentrates in breadth, lower trading frequency, or specific stock buckets.

\begin{figure}[!htbp]
    \centering
    \includegraphics[width=0.83\textwidth]{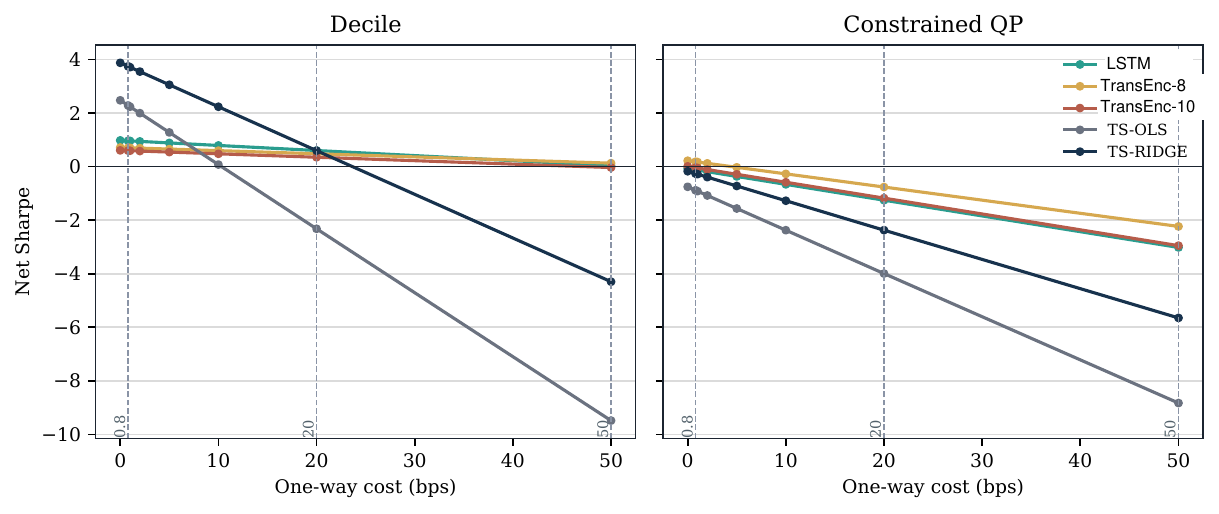}
    \caption{Breakeven transaction-cost diagnostics for promoted models. Curves report net Sharpe under one-way cost schedules for the broad decile layer and the baseline constrained-QP layer; vertical reference lines mark 0.8, 20, and 50 bps.}
    \label{fig:breakevencost}
\end{figure}

Figure~\ref{fig:breakevencost} makes the cost assumption explicit rather than treating 20 bps as a universal breakpoint. A very low commission-like estimate such as 0.8 bps is most relevant for large, liquid names with optimized execution; by contrast, the daily equal-weighted long--short signals in this paper draw much of their strength from low-price and small-stock buckets, where effective half-spreads and market-impact costs can be much larger \citep{DeMiguelMartinUtreraNogalesUppal2020,DetzelNovyMarxVelikov2023,ChenVelikov2023,FrazziniIsraelMoskowitz2018}. In the broad decile layer, the estimated zero-Sharpe costs range from about 10 bps for TS-OLS to roughly 24--61 bps for the other promoted models, so 0.8 bps leaves the decile results positive for most models. In the baseline constrained-QP layer, however, the breakeven costs are at or below a few basis points and are negative for TS-RIDGE, LSTM, and TS-OLS because their constrained gross implementations are already below zero. Appendix~\ref{app:smaa-cost-transfer} reports the corresponding breakeven table, the square-root cost robustness, and a Top1000 low-cost check; together they show that low-cost assumptions help only where a sufficiently liquid universe still contains usable cross-sectional signal.

\begin{table}[!htbp]
\centering
\TableSetup
\caption{Design sensitivity and aggregate boundary tests.}
\label{tab:stabilityboundary}
\input{table_stability_boundary_main.tex}
\end{table}

The F2 TS-RIDGE--LSTM combination weakens under the Top500 restriction, and its large-cap counterpart turns negative. Broad-universe stock sorting and large-cap deployment impose different signal, turnover, and tradability constraints. The full-universe evidence identifies where the signal is strongest, while the large-cap contrast shows how much signal survives a more deployable universe.

\begin{figure}[!htbp]
    \centering
    \begin{subfigure}[t]{0.49\textwidth}
        \centering
        \includegraphics[width=\textwidth]{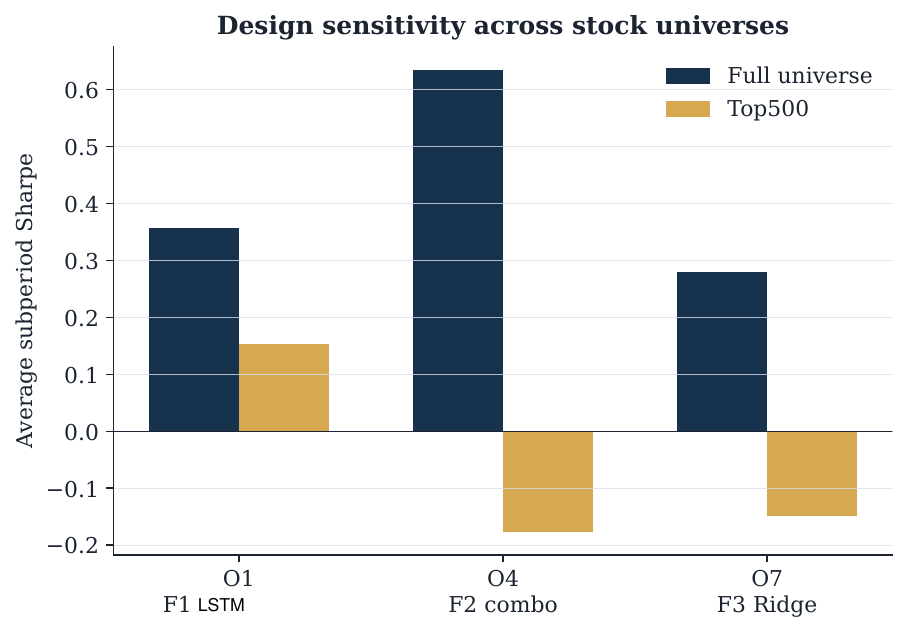}
        \caption{Full-universe versus Top500 design comparison.}
        \label{fig:d12design}
    \end{subfigure}\hfill
    \begin{subfigure}[t]{0.49\textwidth}
        \centering
        \includegraphics[width=\textwidth]{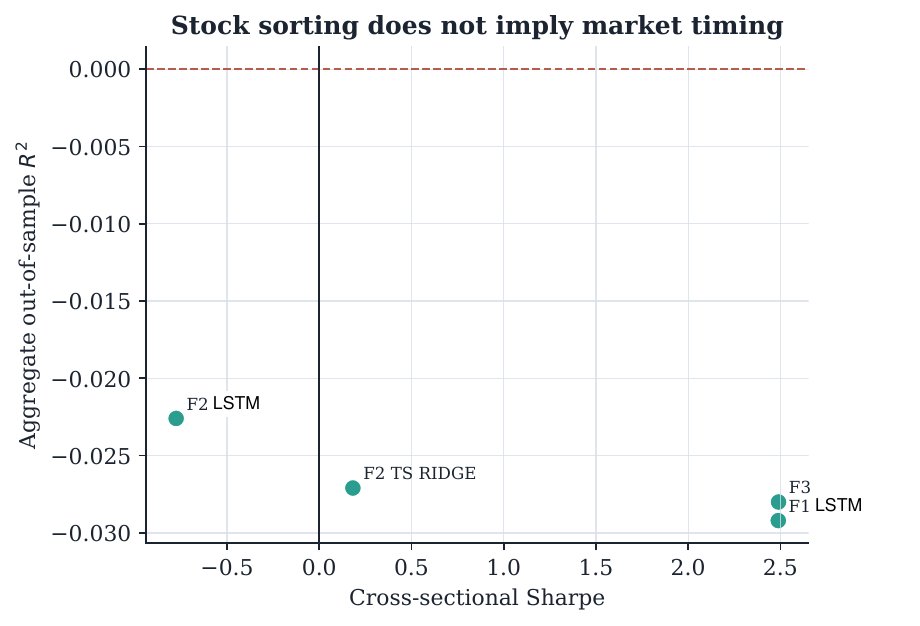}
        \caption{Cross-sectional Sharpe versus aggregate out-of-sample $R^2$.}
        \label{fig:d5boundary}
    \end{subfigure}
    \caption{Design sensitivity and aggregate prediction boundary.}
    \label{fig:designboundarydual}
\end{figure}

The aggregate boundary test is sharper than the state comparison. None of the market-plus-signal specifications produces positive aggregate out-of-sample $R^2$, so strong stock-level sorting does not automatically become market-timing ability. Adjacent forecasting targets impose different constraints once the target, aggregation level, and portfolio rule change.

General-forecasting transfer is retained only as a boundary note and reported in the appendix. The GIFT-Eval proxy comparison is intentionally low powered because exact model matches are rare; it finds no reliable monotone transfer relation from generic time-series forecasting rank to the CRSP cross-sectional decision objective. The absence of transfer is consistent with the paper's main mechanism: portfolio value is mediated by ranks, costs, turnover, and the allocation map rather than by a generic forecasting loss.

\section{Discussion}
\label{sec:discussion}
The results place architecture comparison inside a decision framework. Gu, Kelly, and Xiu~\cite{GuKellyXiu2020} show that flexible machine-learning methods improve monthly cross-sectional prediction, but the daily evidence here is more conditional because the value of architectural capacity depends on signal structure, turnover, and the decision rule that translates scores into positions. The regularization path in Figure~\ref{fig:ridgelambda} extends the shrinkage logic of Kozak, Nagel, and Santosh~\cite{KozakNagelSantosh2020} from factor selection to architecture selection: richer signal sets demand stronger coefficient discipline.

The SMAA layer supplies the operational-research object. It maps each architecture to the preference states under which it is admissible and turns the benchmark into a surface rather than a single ordering. The turnover-weight boundary in Figure~\ref{fig:smaaweightspace} gives an interpretable threshold between return-oriented and cost-sensitive selection, and Proposition~\ref{thm:turnover-alignment} gives the elementary downstream analogue in the turnover penalty of the constrained portfolio. Because the rank-acceptability bands add sampling uncertainty to this preference map, the ranking is read as a decision surface rather than as a fixed ordering.

The constrained portfolio layer changes the economic interpretation. If an institution selected only from the raw SMAA rank-1 surface, TS-RIDGE would remain attractive because of its criteria-level acceptability. After the same scores pass through capacity, beta, industry, risk, and turnover constraints, that choice becomes a costly implementation relative to TransEnc-8. The framework therefore changes the decision: it moves the selected constrained architecture toward the lower-turnover transformer encoder and treats TS-RIDGE as a strong broad-universe sorting signal rather than as the preferred daily constrained-QP implementation. The negative constrained-QP net Sharpe values at 20 bps are retained because they identify this avoided misselection and show that turnover control ranks implementation losses without converting the daily constrained strategy into a profitable standalone rule. Propositions~\ref{thm:regret-lipschitz} and~\ref{thm:robust-shrinkage} provide interpretive optimization links; the realized net-Sharpe regret tables provide the empirical downstream diagnostic.

The capacity and transfer tests define the external-validity boundary. The capacity frontier separates broad-universe stock sorting from large-cap deployability, while the aggregate and general-forecasting transfer checks separate cross-sectional ranking from market timing and generic time-series accuracy. Because the CRSP objective is mediated by ranks, portfolio construction, trading costs, and microstructure, a generic forecasting loss does not track it mechanically.

The implementation checks give a constructive counterpoint to the negative constrained-daily results. Positive net performance appears in broad-universe decile portfolios, selected low-price or small-stock buckets, and weekly rebalancing, so the signal is bounded rather than absent. Model selection should be read as a protocol in which the selected architecture is a function of investor preferences, market state, universe breadth, and implementation frequency.

The limitations are directly testable. The sample spans COVID, the post-COVID recovery, inflation shocks, and elevated cross-sectional dispersion, yet future samples can test additional regimes, and the deliberately compact predictor library leaves room to examine whether broader characteristic sets strengthen or weaken the shrinkage result. The SMAA intervals in this paper cover criteria sampling and preference uncertainty, not training-seed uncertainty in neural estimation; larger compute would allow a further extension outside the present package: multi-seed deep-learning refits. A full day-by-day robust QP reoptimization, rather than the ex-post robust valuation of Table~\ref{tab:drofrontier}, is reported in Section~\ref{sec:portfolio} and Appendix~\ref{tab:supp_robustqpreopt}. A natural operations-research extension would replace uniform Dirichlet preferences with elicited institutional priors and replace the diagonal-risk quadratic program with a richer distributionally robust portfolio layer.

\section{Conclusion}
\label{sec:conclusion}
Benchmarking deep time-series models for daily equity portfolios is a stochastic multi-criteria decision problem. The raw score, the preference-weighted ranking, and the constrained portfolio all matter. A deployment-adjusted acceptability index ties SMAA ranking to downstream regret through an entropically regularized update, and a finite-model turnover-slice result connects the multi-criteria turnover weight to the portfolio turnover penalty. The robustification and sensitivity results provide interpretive optimization bridges for the constrained portfolio layer.

Empirically, the selected architecture is preference- and state-dependent, and the daily signal is real but breadth-bounded: it survives in broad-universe, lower-frequency, and selected-bucket designs and dissolves under large-cap restriction and realistic costs. TransEnc-8 is the least negative constrained daily implementation, while TS-RIDGE retains its strength in broad-universe decile sorting and selected lower-frequency or bucket-restricted designs. The resulting benchmark is a portable selection and diagnosis framework. Its main decision content is the entropically regularized deployment-adjusted index and the turnover-weight to turnover-penalty alignment, while the standard robustness and Lipschitz results serve as interpretive bridges. It is not a standalone trading-strategy claim. Ranking implementation losses is still decision-relevant because it prevents the return-oriented misselection of TS-RIDGE under daily constrained deployment. Future work can replace the preference distribution with institutional priors, strengthen the portfolio layer with a richer distributionally robust portfolio layer, and carry the design to other markets, horizons, and capacity regimes.

\section*{Data and Code Availability}
CRSP data are available through WRDS subject to license restrictions. Fama--French factors are publicly available from the Kenneth R. French Data Library. Appendix~\ref{app:implementation} gives the data sources, feature universes, model settings, and portfolio-layer parameters used in this study. Licensed CRSP inputs are not redistributed.

\newpage

\appendix
\section{Proofs of Methodological Propositions}
\label{app:proofs}

\begin{proof}[Proof of Proposition~\ref{prop:da-entropic}]
On \(S_d\), all components of \(q\) are strictly positive. The objective is a linear term plus the strictly convex relative entropy term, so it is strictly convex on \(\Delta(S_d)\). With a multiplier \(\mu\) for the constraint \(\sum_m p_m=1\), the first-order condition is
\[
\mathcal R_m^d+\tau_d\left(\log\frac{p_m}{q_m}+1\right)+\mu=0,
\qquad m\in S_d.
\]
Thus
\[
p_m=q_m\exp\{-\mathcal R_m^d/\tau_d\}\exp\{-1-\mu/\tau_d\}.
\]
Normalizing over \(S_d\) removes the common constant and gives
\[
p_m=\frac{q_m\exp\{-\mathcal R_m^d/\tau_d\}}{\sum_{\ell\in S_d}q_\ell\exp\{-\mathcal R_\ell^d/\tau_d\}}
=\frac{RA_m(1)\exp\{-\mathcal R_m^d/\tau_d\}}{\sum_{\ell\in S_d}RA_\ell(1)\exp\{-\mathcal R_\ell^d/\tau_d\}}.
\]
This is Definition~4 after writing \(\gamma_d=1/\tau_d\), with zero mass assigned to models outside \(S_d\). Strict convexity gives uniqueness. As \(\tau_d\to\infty\), every exponential term converges to one and the solution returns to \(q\). As \(\tau_d\downarrow0\), the Gibbs weights vanish for all non-minimizers of \(\mathcal R_m^d\) and retain \(q\)-proportional mass on the minimizer set.
\end{proof}

\begin{proof}[Proof of Proposition~\ref{thm:robust-shrinkage}]
For a fixed feasible portfolio \(w\),
\[
\inf_{\tilde s\in \mathcal U_2(\rho)} w^\top\tilde s
=w^\top s+\inf_{\|u\|_2\leq\rho}w^\top u
=w^\top s-\rho\|w\|_2,
\]
where the last equality is the support function of the Euclidean ball applied to \(-w\). Substitution into the max--min objective gives the ellipsoidal robust-allocation identity.

For the distributional statement, the linear map \(\xi\mapsto w^\top\xi\) is \(\|w\|_*\)-Lipschitz under the ground metric induced by \(\|\cdot\|\). Kantorovich--Rubinstein duality therefore gives the lower bound
\[
\E_P[w^\top\xi]\geq \E_{P_0}[w^\top\xi]-\varepsilon\|w\|_*,
\qquad W_1(P,P_0)\leq\varepsilon .
\]
In finite dimension the dual norm is attained: choose \(v\) with \(\|v\|=1\) and \(w^\top v=-\|w\|_*\). Since the ambiguity set has no support restriction, the law \(P_\varepsilon\) of \(\xi+\varepsilon v\), with \(\xi\sim P_0\), is admissible and satisfies \(W_1(P_\varepsilon,P_0)\leq\varepsilon\). For this law,
\[
\E_{P_\varepsilon}[w^\top\xi]
=\E_{P_0}[w^\top\xi]-\varepsilon\|w\|_* .
\]
The lower bound is therefore tight, and
\[
\inf_{P:W_1(P,P_0)\leq \varepsilon}\E_P[w^\top\xi]
=\E_{P_0}[w^\top\xi]-\varepsilon\|w\|_*
=w^\top\bar s-\varepsilon\|w\|_* .
\]
Thus robustification adds a convex norm penalty to the nominal score objective.

For monotonicity in the ellipsoidal radius, fix \(w\). The function \(\rho\mapsto w^\top s-c(w)-\rho\|w\|_2\) is affine and nonincreasing. The robust value is the pointwise supremum of these affine functions over the compact set \(\mathcal W\), hence it is convex and nonincreasing. If \(c\) is strictly convex on the affine hull of \(\mathcal W\), then \(w\mapsto w^\top s-c(w)-\rho\|w\|_2\) is strictly concave for each \(\rho\), so the optimizer is unique. Danskin's theorem gives \(V'(\rho)=-\|w^*(\rho)\|_2\), and since \(V\) is convex its derivative is nondecreasing, so \(\|w^*(\rho)\|_2\) is nonincreasing in \(\rho\). Continuity of \(w^*(\rho)\) follows from Berge's maximum theorem because \(\mathcal W\) is compact and the objective is continuous.
\end{proof}

\begin{proof}[Proof of Proposition~\ref{thm:regret-lipschitz}]
For any \(w\in\mathcal W\),
\[
\left|\{w^\top s-c(w)\}-\{w^\top s'-c(w)\}\right|
=|w^\top(s-s')|
\leq \|w\|_1\|s-s'\|_\infty
\leq L\|s-s'\|_\infty .
\]
If \(w_s\) maximizes \(J_c(s)\), then
\[
J_c(s)-J_c(s')
\leq \{w_s^\top s-c(w_s)\}-\{w_s^\top s'-c(w_s)\}
\leq L\|s-s'\|_\infty .
\]
Reversing \(s\) and \(s'\) gives the fixed-state bound. For different downstream penalty states, the same maximizer argument gives
\[
\begin{aligned}
J_c(s)-J_{\tilde c}(s')
&\leq w_s^\top(s-s')+\tilde c(w_s)-c(w_s) \\
&\leq L\|s-s'\|_\infty+\Delta(c,\tilde c),
\end{aligned}
\]
and reversing the two states gives the absolute-value bound.

For the model-level optimization-value regret under a common downstream state, choose
\[
\ell^*\in\arg\max_{\ell\in\mathcal M_d}J_c(s^{(\ell)}).
\]
Then
\[
\mathcal R_J(m)=J_c(s^{(\ell^*)})-J_c(s^{(m)})
\leq L\|s^{(\ell^*)}-s^{(m)}\|_\infty
\leq L\max_{\ell\in\mathcal M_d}\|s^{(\ell)}-s^{(m)}\|_\infty .
\]
The state-dependent version follows by applying the two-state bound to each pair \((\ell,m)\) and then taking the maximum over \(\ell\). Applying the single-date fixed-state inequality date by date and averaging gives the sequence bound. The result concerns the optimization objective \(J_t\) for a fixed downstream design or for explicitly controlled downstream-state differences. A bound for realized net-Sharpe regret would additionally require assumptions linking the optimized weights to realized returns and controlling the Sharpe denominator; those assumptions are not used in the paper, so the empirical realized-regret tables are interpreted as diagnostics rather than as deterministic corollaries of this proposition.
\end{proof}
\begin{proof}[Proof of Proposition~\ref{thm:turnover-alignment}]
For any two architectures \(i,j\), the preference-layer utility difference is
\[
U_i(u)-U_j(u)
=(1-u)(a_i-a_j)+u(c_i-c_j).
\]
Using \(c_m=b-\kappa\tau_m\), this becomes
\[
U_i(u)-U_j(u)
=(1-u)(a_i-a_j)-u\kappa(\tau_i-\tau_j)
=(1-u)\{(a_i-a_j)-\lambda(u)(\tau_i-\tau_j)\}.
\]
Since \(1-u>0\) for \(u\in[0,1)\), the sign of this difference is the sign of
\[
\Pi_i(\lambda(u))-\Pi_j(\lambda(u))
=(a_i-a_j)-\lambda(u)(\tau_i-\tau_j).
\]
Therefore every pairwise comparison agrees under \(U(\cdot,u)\) and \(\Pi(\cdot,\lambda(u))\), and the full ranking over the finite set \(\mathcal M\) agrees as well, including ties.

If \(a_i>a_j\) and \(\tau_i>\tau_j\), the common switch is obtained by setting the displayed difference to zero:
\[
u^*_{ij}=\frac{a_i-a_j}{(a_i-a_j)+\kappa(\tau_i-\tau_j)}.
\]
The portfolio proxy switch solves \((a_i-a_j)-\lambda(\tau_i-\tau_j)=0\), giving \(\lambda^*_{ij}=(a_i-a_j)/(\tau_i-\tau_j)\). Direct substitution yields
\[
\kappa\frac{u^*_{ij}}{1-u^*_{ij}}
=\kappa\frac{a_i-a_j}{\kappa(\tau_i-\tau_j)}
=\frac{a_i-a_j}{\tau_i-\tau_j}
=\lambda^*_{ij}.
\]
As \(u\uparrow1\), \(\lambda(u)\to\infty\), and the non-turnover aggregate \((1-u)a_m\) vanishes relative to the turnover coordinate, giving the stated turnover-only limiting ranking.
\end{proof}

\section{Feature and Combination Diagnostics}
\label{app:features-combos}
\label{app:features}
The appendix tables report the feature definitions, distributional checks, ablation results, and model-combination diagnostics behind the main text. Together they document the signal design and robustness checks behind the main ranking.

{\TableSetup\input{table_feature_definitions_longtable.tex}}

Table~\ref{tab:features} defines the twenty-four predictors and maps each predictor to the F1, F2, and F3 signal universes. The table separates price-path variables, activity variables, and size/beta/market transforms, which is the basis for the ablation design.

\begin{table}[!htbp]
\centering
\TableSetup
\caption{Summary statistics for the twenty-four CRSP predictors. Statistics are computed from the full F3 design panel after the same CRSP filters and rolling feature construction used in the benchmark.}
\label{tab:featurestats}
\FitTable{\input{table_feature_summary_statistics.tex}}
\end{table}

Table~\ref{tab:featurestats} shows that the retained predictors have complete post-screening coverage. The percentile columns show that the daily equity features keep wide cross-sectional tails after filtering.

{\TableSetup\input{table_feature_block_ablation_longtable.tex}}

Table~\ref{tab:ablationappendix} gives the full feature-block ablation under the common forecast protocol. Its F3 rows agree with the confirmatory benchmark, and the appendix audit tables report the 10-decile return matrix and the RankIC sign-direction checks used to audit the table.

\begin{table}[!htbp]
\centering
\TableSetup
\caption{Decile-by-decile next-day average returns for the promoted models across F1, F2, and F3. Scores use the same direction as the RankIC calculation: high raw score maps to D10, low raw score maps to D1, and the long--short spread is D10-D1.}
\label{tab:supp_decile_return_matrix}
\FitTable{\input{table_decile_return_matrix_reviewfix.tex}}
\end{table}

\begin{figure}[!htbp]
    \centering
    \begin{subfigure}{0.32\textwidth}
        \includegraphics[width=\linewidth]{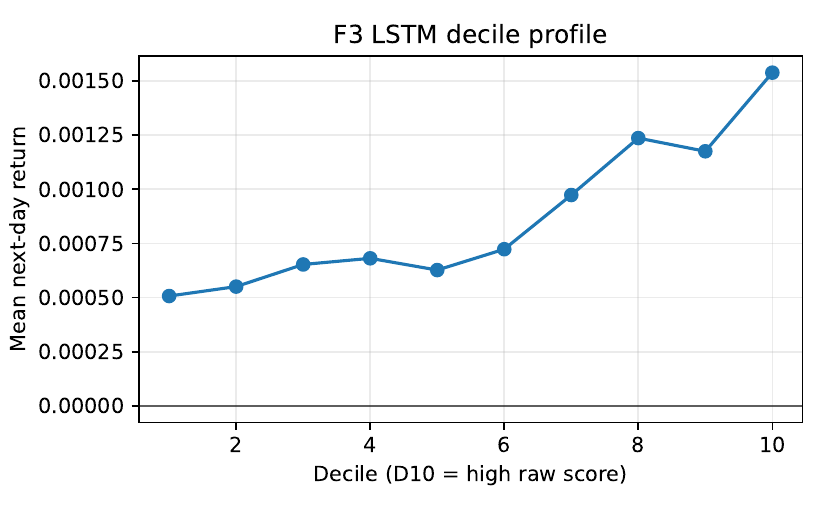}
        \caption{LSTM}
    \end{subfigure}\hfill
    \begin{subfigure}{0.32\textwidth}
        \includegraphics[width=\linewidth]{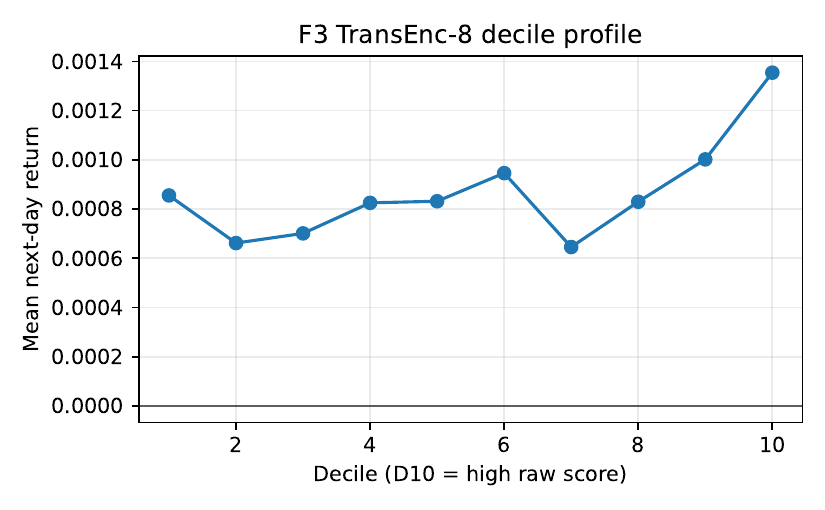}
        \caption{TransEnc-8}
    \end{subfigure}\hfill
    \begin{subfigure}{0.32\textwidth}
        \includegraphics[width=\linewidth]{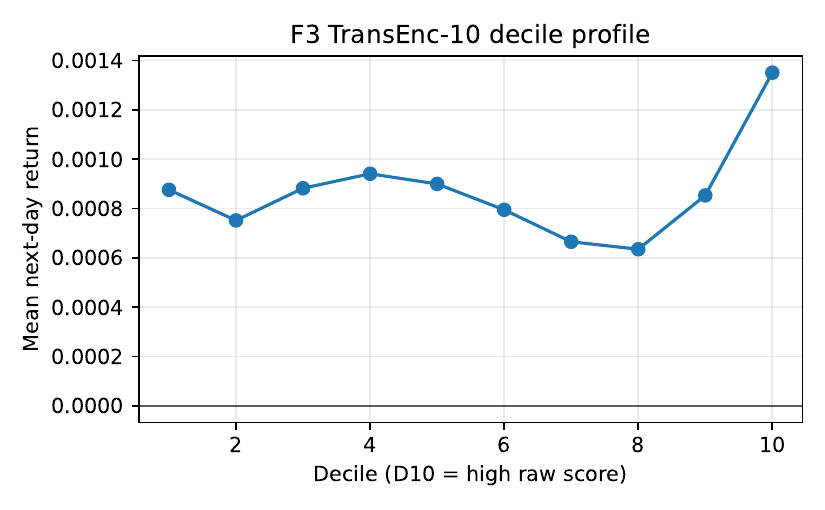}
        \caption{TransEnc-10}
    \end{subfigure}
    \caption{F3 neural-model decile-return profiles used to audit the negative RankIC and positive long--short Sharpe cases.}
    \label{fig:supp_f3_neural_decile_profiles}
\end{figure}

\begin{figure}[!htbp]
    \centering
    \begin{subfigure}{0.32\textwidth}
        \includegraphics[width=\linewidth]{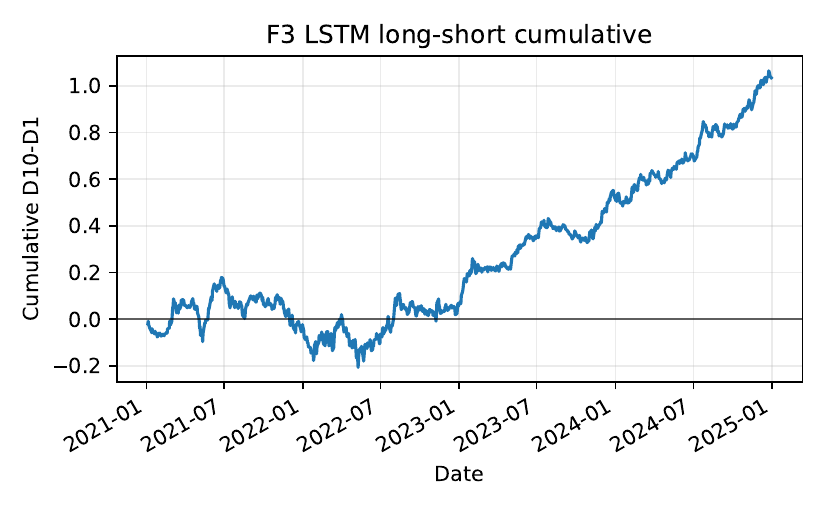}
        \caption{LSTM}
    \end{subfigure}\hfill
    \begin{subfigure}{0.32\textwidth}
        \includegraphics[width=\linewidth]{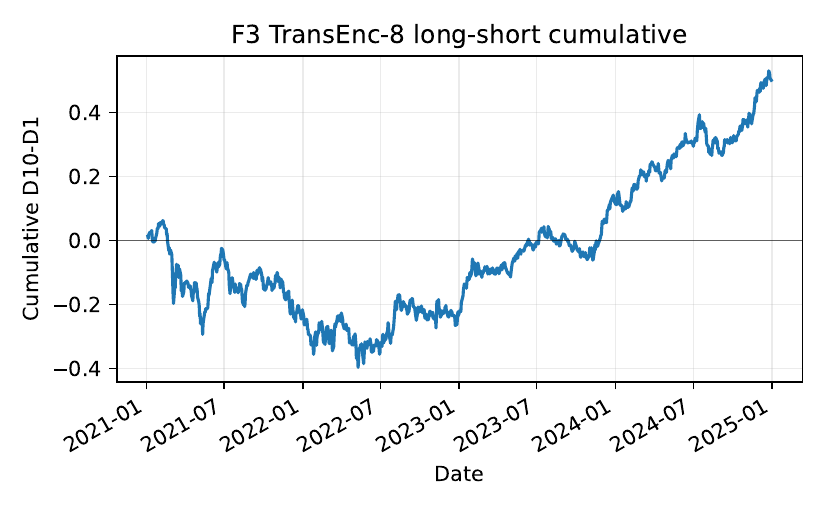}
        \caption{TransEnc-8}
    \end{subfigure}\hfill
    \begin{subfigure}{0.32\textwidth}
        \includegraphics[width=\linewidth]{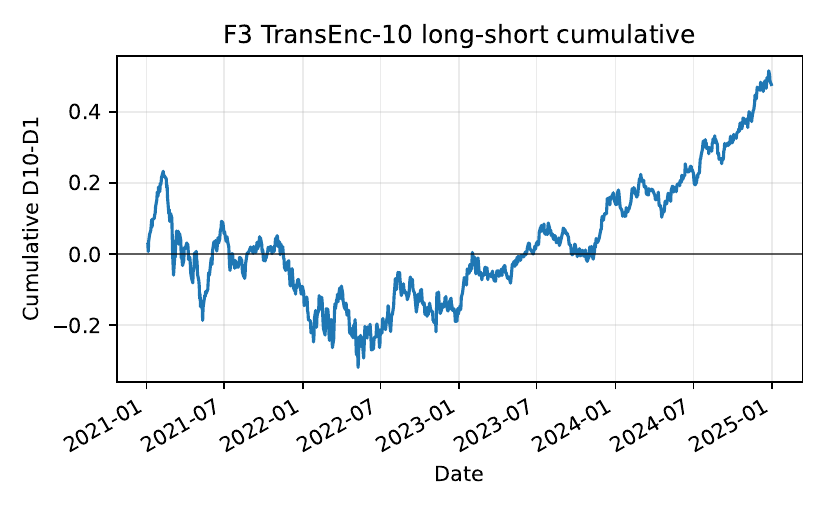}
        \caption{TransEnc-10}
    \end{subfigure}
    \caption{F3 cumulative long--short returns under the same score direction as the RankIC calculation.}
    \label{fig:supp_f3_neural_cumulative}
\end{figure}

The sign audit uses one convention throughout: high raw score enters the D10 long leg, low raw score enters the D1 short leg, and the spread is D10-D1. Under that convention, the F3 neural rows have negative average daily RankIC but positive long--short Sharpe, so the discrepancy is a tail and non-monotonicity effect rather than a long--short direction or label convention error.

\FloatBarrier
\newpage
\begin{table}[!htbp]
\centering
\TableSetup
\caption{Promoted-set pairwise bootstrap Sharpe-difference tests.}
\label{tab:pairwisepromoted}
\FitTable{\input{table_promoted_pairwise_sharpe.tex}}
\end{table}

Table~\ref{tab:pairwisepromoted} gives pairwise Sharpe differences for the promoted set. The gross-return tests separate TS-RIDGE from most peers, while the cost-adjusted analysis in the main text gives the economic ranking.

\begin{table}[!htbp]
\centering
\TableSetup
\caption{Full GIFT-Eval proxy mapping for the retained model universe.}
\label{tab:giftmapfull}
\FitTable{\input{table_general_finance_mapping.tex}}
\end{table}

Table~\ref{tab:giftmapfull} gives the external forecasting match used in the general-forecasting transfer note in the main text. Exact matches are limited after relabeling to the actual implementations, so the table gives the proxy tier used for each architecture and does not treat implementation aliases as exact matches.

% GIFT-Eval scatterplot omitted from the arXiv package; low-power proxy-transfer diagnostics are reported in Table~\ref{tab:giftstats}.

\begin{table}[!htbp]
\centering
\TableSetup
\caption{Promoted-set pairwise rolling-combination matrix.}
\label{tab:combomatrixappendix}
\FitTable{\input{table_promoted_combo_matrix.tex}}
\end{table}

Table~\ref{tab:combomatrixappendix} gives signal correlation, return correlation, and combination performance. Low signal correlation is useful only when both legs retain economic value after costs.

\newpage

\section{SMAA, Cost, and Transfer Diagnostics}
\label{app:smaa-cost-transfer}
This section keeps diagnostics that support the main text but are not needed for the main decision narrative.

\begin{table}[!htbp]
\centering
\TableSetup
\caption{Monte Carlo precision check for SMAA rank-1 acceptability. The 50,000 preference draws are used only for this convergence diagnostic; the main SMAA design uses 10,000 moving-block criteria resamples and 10,000 Dirichlet preference draws.}
\label{tab:supp_smaaconvergence}
\FitTable{\input{table_v8_smaa_convergence.tex}}
\end{table}

\begin{figure}[!htbp]
    \centering
    \includegraphics[width=0.70\textwidth]{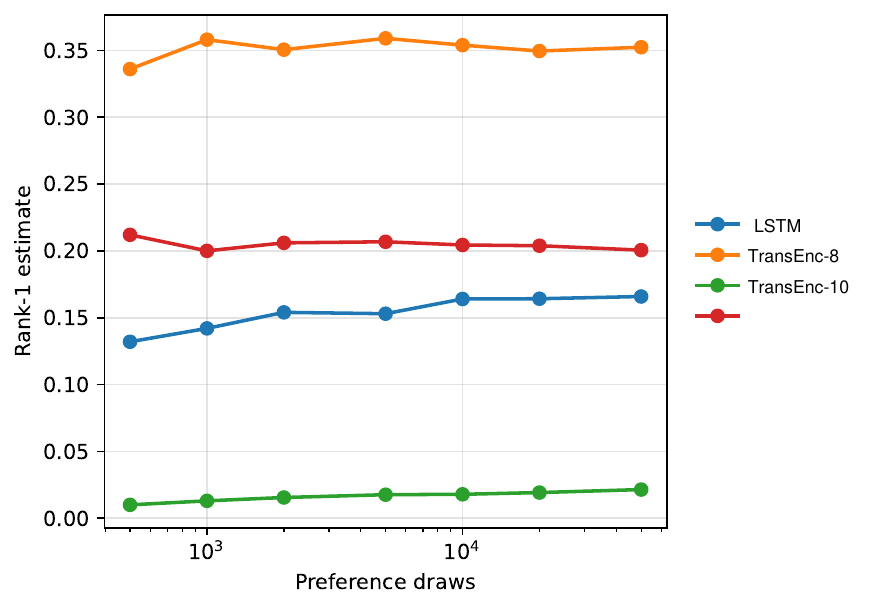}
    \caption{SMAA rank-1 Monte Carlo convergence for leading models.}
    \label{fig:supp_smaaconvergence}
\end{figure}

\begin{table}[!htbp]
\centering
\TableSetup
\caption{Deployment-adjusted rank-1 sensitivity to the dimensionless regret-discount multiplier \(c\). The production definition uses \(c=1\); the grid is an audit-only perturbation of the exponent \(\exp\{-c\mathcal R_m/\tau_d\}\), with \(\tau_d\approx1.170\).}
\label{tab:supp_da_gamma}
\FitTable{\input{table_deployment_adjusted_sensitivity.tex}}
\end{table}

\begin{figure}[!htbp]
    \centering
    \includegraphics[width=0.72\textwidth]{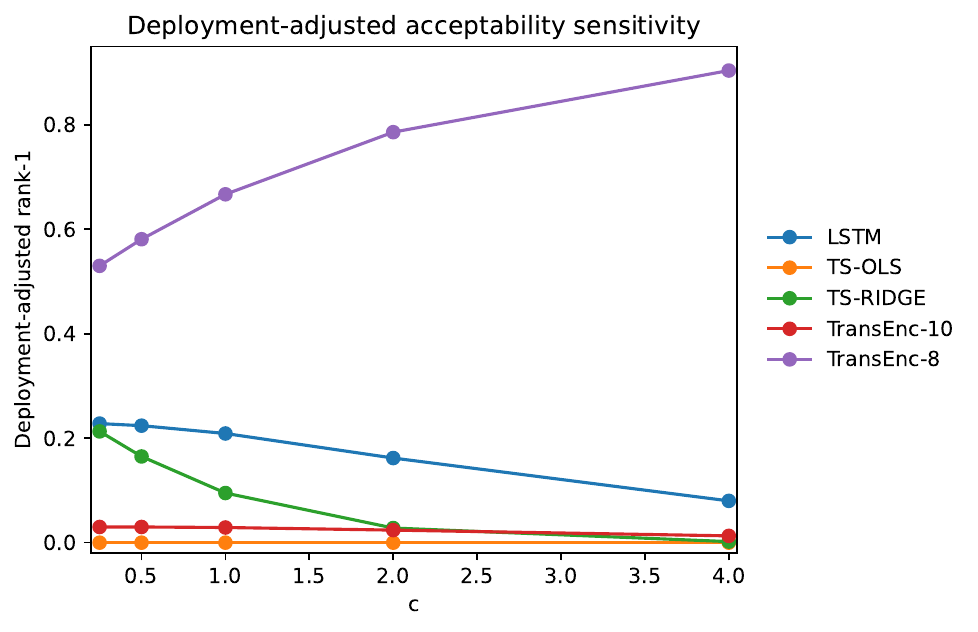}
    \caption{Deployment-adjusted rank-1 sensitivity across dimensionless regret-discount multipliers \(c\).}
    \label{fig:supp_da_gamma}
\end{figure}

\begin{table}[!htbp]
\centering
\TableSetup
\caption{Day-by-day robust QP re-optimization frontier. Entries are realized net Sharpe ratios at 20 bps after re-solving the robust constrained portfolio each day.}
\label{tab:supp_robustqpreopt}
\FitTable{\input{table_robust_qp_reoptimized_frontier.tex}}
\begin{minipage}{0.92\textwidth}
\footnotesize \emph{Note.} \(\rho=0\) reproduces the Table 13 baseline net Sharpe to within OSQP re-solve tolerance \((<0.03\) Sharpe). The \(L_2\) robust penalty is solved by an MM/IRLS sequence of OSQP quadratic subproblems with at most eight iterations and tolerance \(10^{-5}\).
\end{minipage}
\end{table}

\begin{figure}[!htbp]
    \centering
    \includegraphics[width=0.72\textwidth]{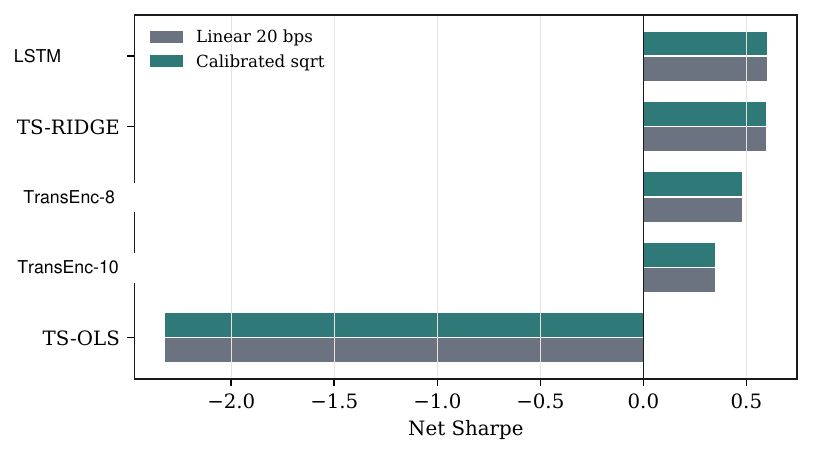}
    \caption{Linear and square-root transaction-cost robustness for promoted common-window portfolios.}
    \label{fig:supp_sqrtcost}
\end{figure}

\begin{table}[!htbp]
\centering
\TableSetup
\caption{Breakeven one-way transaction costs by promoted model and decision layer. Negative breakeven costs indicate that the corresponding gross constrained-QP implementation is already below zero before costs.}
\label{tab:supp_breakevencost}
\FitTable{\input{table_breakeven_cost_by_model_layer.tex}}
\end{table}

\begin{table}[!htbp]
\centering
\TableSetup
\caption{Top1000 large-cap low-cost check. Net Sharpe ratios are computed for the same corrected capacity daily series under one-way costs of 0.8, 1, 2, and 20 bps.}
\label{tab:supp_top1000lowcost}
\FitTable{\input{table_large_cap_low_cost_top1000.tex}}
\end{table}

Table~\ref{tab:supp_breakevencost} separates the common-window decile layer, the baseline constrained-QP layer, and the low-price/small-stock buckets. The broad decile portfolios survive substantially larger costs than the constrained-QP implementations, while the bucket rows show that much of the cost tolerance is concentrated in small or low-price names. Table~\ref{tab:supp_top1000lowcost} gives the large-cap low-cost counterfactual. Even at 0.8--2 bps, most Top1000 daily decile portfolios remain weak, with only TransEnc-8 and TransEnc-10 slightly positive; the result reinforces the main text's interpretation that low trading costs help only when the liquid universe still contains enough cross-sectional signal.

\begin{table}[!htbp]
\centering
\TableSetup
\caption{Core GIFT-Eval correlation diagnostics. The comparison is retained as a low-power transfer note rather than as a model-ranking input.}
\label{tab:supp_giftstats}
\FitTable{\input{table_general_finance_stats_core.tex}}
\end{table}

\FloatBarrier

\section{Tradability and Stability Diagnostics}
\label{app:tradability-stability}
The tradability tables report bucket performance, turnover decomposition, concentration, and subperiod splits. These diagnostics identify which results are broad and which depend on the trading universe.

\begin{table}[!htbp]
\centering
\TableSetup
\caption{Bucket-level net performance for the central promoted models and the F2 Ridge--LSTM combo.}
\label{tab:bucketappendix}
\FitTable{\input{table_bucket_net_performance_appendix.tex}}
\end{table}

Table~\ref{tab:bucketappendix} reports the central bucket diagnostics for TS-RIDGE, LSTM, and the F2 Ridge--LSTM combo. These rows are enough to show where trading frictions and stock characteristics concentrate the losses and avoids a duplicate model-by-model supplementary listing.

{\TableSetup\input{table_turnover_overlap.tex}}

Table~\ref{tab:turnoveroverlap} decomposes one-way turnover into within-leg rebalancing, entry/exit turnover, and cross-side migration. This table explains why models with similar gross returns can have different net Sharpe ratios.

\begin{table}[!htbp]
\centering
\TableSetup
\caption{Quintile-portfolio robustness for the leading common-window models. Quintile portfolios replace the benchmark decile rule with top and bottom fifth portfolios and retain the same common dates and turnover reconstruction.}
\label{tab:quintile}
\FitTable{\input{table_quintile_robustness.tex}}
\end{table}

Table~\ref{tab:quintile} checks whether the results depend on the decile cutoff. The quintile portfolios lower concentration and preserve the main ordering among the leading models.

\begin{table}[!htbp]
\centering
\TableSetup
\caption{Daily and weekly rebalancing robustness for promoted models.}
\label{tab:weeklyrebalancing}
\FitTable{\input{table_weekly_rebalance.tex}}
\end{table}

\begin{figure}[!htbp]
    \centering
    \includegraphics[width=0.74\textwidth]{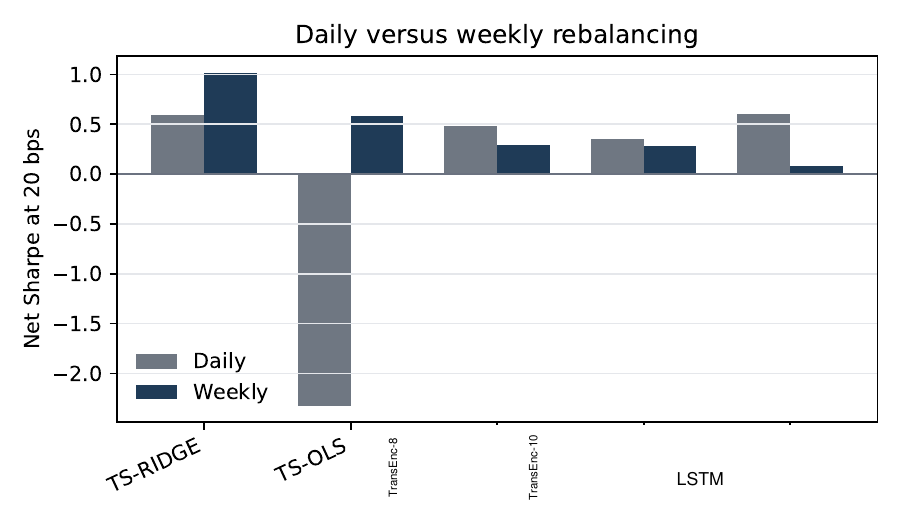}
    \caption{Net Sharpe ratios under daily and weekly rebalancing.}
    \label{fig:weeklyrebalancing}
\end{figure}

Table~\ref{tab:weeklyrebalancing} and Figure~\ref{fig:weeklyrebalancing} compare daily and five-day rebalanced decile portfolios. Weekly rebalancing cuts turnover sharply for every promoted model. TS-RIDGE moves from daily net Sharpe of 0.60 to weekly net Sharpe of 1.02, and TS-OLS also improves because trading intensity falls. The nonlinear models have lower turnover under both rules, so the weekly change is smaller for them. The ranking under weekly rebalancing is led by TS-RIDGE, TS-OLS, and TransEnc-8.

\FloatBarrier
\begin{table}[!htbp]
\centering
\TableSetup
\caption{State-by-subperiod stability summaries.}
\label{tab:stabilitystateappendix}
\FitTable{\input{table_direction12_state.tex}}
\end{table}

Table~\ref{tab:stabilitystateappendix} splits the central portfolios by volatility state and subperiod. The table shows whether each result comes from one state or persists across market conditions.

\begin{table}[!htbp]
\centering
\TableSetup
\caption{Design-sensitivity subperiod summaries.}
\label{tab:stabilitydesignappendix}
\FitTable{\input{table_direction12_design.tex}}
\end{table}

Table~\ref{tab:stabilitydesignappendix} compares full-universe and Top500 results by subperiod. The table separates broad stock-sorting performance from large-cap deployability.

\FloatBarrier
\newpage
\section{Implementation and Reproducibility}
\label{app:implementation}
The implementation tables summarize the data, training, feature-universe, and model-setting choices used in the benchmark. Licensed raw inputs are excluded, and the replication design is fully specified.

The empirical analysis used a fixed software environment: Python 3.12.7, NumPy 1.26.4, SciPy 1.13.1, pandas 2.2.2, Matplotlib 3.9.2, scikit-learn 1.5.1, statsmodels 0.14.2, \texttt{arch} 8.0.0, CVXPY 1.6.5, OSQP 1.1.1, and PyTorch 2.8.0. Forecast estimation, resampling, financial diagnostics, portfolio construction, and optimization were run from fixed scripts with deterministic data splits.

Most computation came from repeated walk-forward model estimation and common-window forecast construction. The remaining computations covered feature-universe comparisons, financial diagnostics, stochastic multi-criteria analysis, Model Confidence Set estimation, capacity-frontier calculations, transaction-cost sensitivity analysis, and constrained portfolio optimization. Random seeds were fixed before stochastic model estimation and resampling, with the exact values recorded in the replication scripts rather than used as substantive results.

The financial comparison uses fixed stock-level forecasts produced by one training protocol. These appendices specify data provenance, estimation settings, feature universes, model settings, and decision-layer parameters. Licensed CRSP inputs are excluded from the arXiv package.

\begin{table}[!htbp]
\centering
\TableSetup
\caption{Data and replication protocol.}
\label{tab:datacodemanifest}
\FitTable{\input{table_data_code_manifest.tex}}
\end{table}

Table~\ref{tab:datacodemanifest} lists the data sources and replication components. It gives the ingredients needed to rebuild the analysis while keeping licensed CRSP inputs outside the arXiv package.

\begin{table}[!htbp]
\centering
\TableSetup
\caption{Detailed common training, alignment, inference, and portfolio protocol.}
\label{tab:trainingprotocol}
\FitTable{\input{table_training_protocol_detailed.tex}}
\end{table}

Table~\ref{tab:trainingprotocol} fixes the common filters, alignment, inference method, and portfolio construction rule. These choices define the shared evaluation window used across models.

\begin{table}[!htbp]
\centering
\TableSetup
\caption{Feature-universe implementation protocol for confirmatory and ablation analyses.}
\label{tab:featureuniverseprotocol}
\FitTable{\input{table_feature_universe_protocol.tex}}
\end{table}

Table~\ref{tab:featureuniverseprotocol} maps F1, F2, F3, and the leave-one-block-out variants. It gives the signal sets used in the mechanism tests.

\begin{table}[!htbp]
\centering
\TableSetup
\caption{Model label mapping. Source registry keys are early scaffolding labels inherited from the experiment harness and are kept unchanged in the replication scripts; they bear no relation to model selection or to the architectures actually estimated. The manuscript labels describe the actual canonical implementation used in the benchmark, not the frontier system whose name a registry key happens to reuse.}
\label{tab:modellabelmapping}
\FitTable{\input{table_model_label_mapping.tex}}
\end{table}

\begin{table}[!htbp]
\centering
\TableSetup
\caption{Model implementation settings for the retained fifteen-model benchmark. Parameter counts use the full F3 input tensor with 60 time steps and 24 features.}
\label{tab:modelhyperparams}
\FitTable{\input{table_model_hyperparameters.tex}}
\end{table}

Table~\ref{tab:modelhyperparams} summarizes the estimator or neural architecture used for each retained model. The parameter counts show how model capacity varies within the common protocol.

\subsection*{SMAA, MCS, and QP settings}
SMAA uses seven criteria: gross Sharpe, net Sharpe at 20 bps, value-weighted Sharpe, absolute five-factor alpha $t$-statistic, negative turnover, negative maximum drawdown, and bootstrap significance. Criteria uncertainty uses 10,000 moving-block resamples. Preference uncertainty uses 10,000 Dirichlet weight draws on the seven-dimensional simplex.

The Model Confidence Set uses moving-block bootstrap inference with 10,000 replications and block size 21. Gross and net return matrices are tested at 90\% and 95\% confidence. The inputs are the same 1,197 common evaluation dates used in the master benchmark.

The constrained portfolio problem is solved with OSQP. Each promoted model is optimized over 1,197 daily decisions under dollar neutrality, leverage of 2, a 1\% dollar-volume single-name capacity cap for a 100 million dollar portfolio, beta neutrality, industry net exposure limits, and a diagonal risk estimate, with portfolio PnL measured from next-trading-day realized stock returns. Forecast scores are centered and z-scored within each trading-day cross-section before they enter the linear score term of the QP, so the optimization uses standardized ranking signals rather than raw model-output scale. The implementation designs are baseline, higher turnover penalty, and 50 bps cost stress; the low-risk variant is omitted from the main table because it is numerically indistinguishable from the baseline in this sample.

Regime-conditioned SMAA repeats the seven-criterion acceptability analysis separately on high- and low-volatility dates. The sequential validation check estimates acceptability on the first half of the evaluation period and compares the resulting ranking with the second half. The weekly rebalancing check keeps the daily model scores fixed and changes only the portfolio holding schedule to five trading days.

Preference-space geometry uses the same seven normalized criteria as the main SMAA calculation and draws 10,000 Dirichlet weight vectors. The structured-prior check uses return-oriented, net-performance-oriented, and statistical-rigor Dirichlet concentration vectors. The optimized-portfolio SMAA uses QP gross Sharpe, net Sharpe, annualized net return, turnover, maximum drawdown, and bootstrap significance from the optimized daily return series.

Automatic block-length diagnostics for the central daily long--short return series follow the Politis--White automatic block-length rule~\citep{PolitisWhite2004}. The rounded circular estimates are 3 for TS-RIDGE, 1 for LSTM, and 2 for TransEnc-8, with stationary estimates of 2.4, 0.9, and 1.8, respectively. These selected lengths are below the 20-day benchmark block. The main inference keeps the longer 20-day block to avoid understating serial dependence in a daily equity setting.

\bibliographystyle{unsrtnat}
\bibliography{references}

\end{document}

%% file: table_cost_sensitivity_sharpe.tex
\begin{tabular}{lrrrrr}
\toprule
Model & 0bps & 10bps & 20bps & 30bps & 50bps \\
\midrule
TS-RIDGE & 3.88 & 2.23 & 0.60 & -1.04 & -4.30 \\
LSTM & 0.98 & 0.79 & 0.60 & 0.41 & 0.03 \\
TransEnc-8 & 0.71 & 0.60 & 0.48 & 0.36 & 0.13 \\
TransEnc-10 & 0.61 & 0.48 & 0.35 & 0.22 & -0.04 \\
TS-OLS & 2.48 & 0.08 & -2.32 & -4.72 & -9.49 \\
\bottomrule
\end{tabular}

%% file: table_v8_dro_frontier.tex
\begin{tabular}{lrrrrr}
\toprule
Model & $\rho=0$ & $\rho=0.001$ & $\rho=0.004$ & Avg. $\ell_2$ & Days \\
\midrule
TS-RIDGE & -2.37 & -2.58 & -3.19 & 0.377 & 1197 \\
LSTM & -1.25 & -1.48 & -2.17 & 0.383 & 1197 \\
TransEnc-8 & -0.76 & -0.99 & -1.68 & 0.436 & 1197 \\
TransEnc-10 & -1.18 & -1.34 & -1.83 & 0.262 & 1197 \\
TS-OLS & -3.99 & -4.20 & -4.82 & 0.353 & 1197 \\
\bottomrule
\end{tabular}

%% file: table_deployment_adjusted_acceptability.tex
\begin{tabular}{lrrrr}
\toprule
Model & SMAA rank-1 & QP regret & Deployment adj. rank-1 & Deployment adj. top-3 \\
\midrule
TransEnc-8 & 0.352 & 0.000 & 0.667 & 0.430 \\
LSTM & 0.168 & 0.492 & 0.209 & 0.282 \\
TS-RIDGE & 0.199 & 1.611 & 0.095 & 0.046 \\
TransEnc-10 & 0.022 & 0.414 & 0.029 & 0.241 \\
TS-OLS & 0.000 & 3.231 & 0.000 & 0.000 \\
\bottomrule
\end{tabular}

%% file: table_risk_diagnostics_dual_panel.tex
\textit{Panel A. Market-proxy diagnostics.}\par\smallskip
\FitTable{\input{table_factor_alpha_main.tex}}
\medskip
\textit{Panel B. Holdings exposure diagnostics.}\par\smallskip
\FitTable{\input{table_holdings_exposure_main.tex}}

%% file: table_stability_boundary_main.tex
\textit{Panel A. Full-stock versus Top500 design sensitivity.}\par\smallskip
\FitTable{\input{table_direction12_main.tex}}
\medskip
\textit{Panel B. Aggregate market-timing boundary.}\par\smallskip
\FitTable{\input{table_direction5_main.tex}}

%% file: table_v8_smaa_convergence.tex
\begin{tabular}{lrr}
\toprule
Model & Final rank-1 & 95\% half-width \\
\midrule
TransEnc-8 & 0.352 & 0.006 \\
TS-RIDGE & 0.199 & 0.006 \\
LSTM & 0.168 & 0.006 \\
TransEnc-10 & 0.022 & 0.006 \\
\bottomrule
\end{tabular}

%% file: table_deployment_adjusted_sensitivity.tex
\begin{tabular}{lrrrrr}
\toprule
Model & $c=0.25$ & $c=0.50$ & $c=1.00$ & $c=2.00$ & $c=4.00$ \\
\midrule
LSTM & 0.228 & 0.224 & 0.209 & 0.162 & 0.080 \\
TS-OLS & 0.000 & 0.000 & 0.000 & 0.000 & 0.000 \\
TS-RIDGE & 0.213 & 0.165 & 0.095 & 0.028 & 0.002 \\
TransEnc-10 & 0.030 & 0.030 & 0.029 & 0.024 & 0.013 \\
TransEnc-8 & 0.530 & 0.581 & 0.667 & 0.786 & 0.904 \\
\bottomrule
\end{tabular}

%% file: table_robust_qp_reoptimized_frontier.tex
\begin{tabular}{lrrr}
\toprule
Model & $\rho=0$ & $\rho=0.001$ & $\rho=0.004$ \\
\midrule
TransEnc-8 & -0.76 & -0.76 & -0.76 \\
TransEnc-10 & -1.18 & -1.18 & -1.18 \\
LSTM & -1.25 & -1.25 & -1.26 \\
TS-RIDGE & -2.38 & -2.38 & -2.37 \\
TS-OLS & -3.99 & -4.00 & -4.02 \\
\bottomrule
\end{tabular}